\documentclass[11pt] {amsart}

\usepackage{verbatim}
\usepackage{amssymb}
\usepackage{amsbsy}
\usepackage{amscd}
\usepackage{amsmath}
\usepackage{amsthm}
\usepackage[mathscr]{eucal}
\makeatletter
\newtheorem{Theorem}{Theorem}%[section]
\newtheorem{Lojasiewicz Join Theorem}[Theorem]{\L ojasiewicz Join Theorem}%[section]
\newtheorem{Lemma}[Theorem]{Lemma}

\newtheorem{Proposition}[Theorem]{Proposition}

\newtheorem{Definition}[Theorem]{Definition}
\newtheorem{Remark}[Theorem]{Remark}

\newtheorem{Example}[Theorem]{Example}

\newtheorem{ Observation}[Theorem]{Observation}
\newtheorem{Assertion}[Theorem]{Assertion}

\newcommand{\eps}{\varepsilon}

\newcommand\la{\lambda}
\newcommand\vphi{\varphi}

\newcommand\be{\beta}

\newcommand\ga{\gamma}
\newcommand\Ga{\Gamma}
\newcommand\de{\delta}
\newcommand\De{\Delta}

\newcommand\bfu{\mbox {\bf  u}}

\newcommand\bfv{\mbox {\bf  v}}

\newcommand\bfw{\mbox {\bf  w}}

\newcommand\bfz{\mbox {\bf  z}}

\newcommand\bfa{\mbox {\bf  a}}

\newcommand\nlind{\nl \indent}
\newcommand\nl{\newline}

\newcommand{\ord}{\rm{ord}\/}

\newcommand\Cone{\rm{Cone}\/}

\newcommand\Coeff{\rm{Coeff}\/}

\newcommand\rdeg{{\rm{rdeg}\/}}

\newcommand\inv{^{-1}}

\def\inv{^{-1}}

\begin{document}
\title[\L ojasiewicz exponents f non-degenerate functions]
{\L ojasiewicz exponents of  non-degenerate holomorohic and mixed functions}

\author
%Normally smooth divisors  {\it Draft: \today}
[M. Oka ]
{Mutsuo Oka }
\address{Department of Mathematics,
{Tokyo  University of Science}
%{Kagurazaka 1-3, Shinjuku-ku}
%{Tokyo 162-8601}
}
\email { oka@rs.kagu.tus.ac.jp}
%\today
%\thanks{}
\keywords {Lojasiewicz inequality,vanishing coordinate, non-convenient}
\subjclass[2000]{14P05,32S55}

\begin{abstract}
% along the exceptional divisors of $\pi$ (\cite{Oka
We consider  \L ojasiewicz inequalities for a non-degenerate  holomorphic function with an isolated singularity at the origin.
We give an explicit estimation of the \L ojasiewicz exponent in a slightly weaker form than the assertion in Fukui \cite{Fukui}.
For a weighted homogeneous polynomial, we give a better estimation in the form   which is conjectured by \cite{B-K-O} under under some condition (the \L ojasiewicz non-degeneracy).
 We also introduce \L ojasiewicz inequality for strongly non-degenerate mixed functions and   generalize this estimation
for  mixed functions.
\end{abstract}
\maketitle

\maketitle

%\section{Introduction}
\section{ Holomorphic functions
 and \L ojasiewicz exponent}
Consider an analytic function $f(\bfz)$ with an isolated singularity at the origin. 
We consider the inequality
\begin{eqnarray}\label{Loj-ineq}
\|\partial f(\bfz)\|\ge c\|\bfz\|^\theta,\quad \exists c>0,\forall \bfz\in U
\end{eqnarray}
where $U$ is a sufficiently small neighborhood of the origin and $\partial f(z)$ is the gradient vector
$(\frac{\partial f}{\partial z_1},\dots,\frac{\partial f}{\partial z_n})$.
The \L ojasiewicz exponent $\ell_0(f)$ of $f(\bfz)$ at the origin is the smallest 
positive number among  $\theta$'s which satisfy the inequality (\ref{Loj-ineq}).
It is known that there exists such number $\ell_0(f)$ and it is a rational number \cite{L-T,Te}.
We assume that the Newton boundary of $f$ is non-degenerate hereafter.
The purpose of this paper is to give an explicit upper bound of the \L ojasiewicz exponent in term of the combinatorics of the Newton boundary.
There is a similar estimation proposed by  Fukui \cite{Fukui} but he uses some incorrect equality (2.4), \cite{Fukui} in his  proof.  
Thus the proof has a gap and   the assertion must be proved in a different way, even if it is true. 
There is also an estimation by Abderrahmane \cite{Ab2} using Newton number. In this paper, we give an estimation along Fukui's way.
Our estimation is apparently a little weaker than that of Fukui but it is enough for our purpose.
In the last section of this paper (\S 4), we will generalize the notion  of \L ojasiewicz exponent for mixed functions. 
%%%%%
\subsection{Newton boundary and the dual Newton diagram}
Let $f(\bfz)=\sum_{\nu}c_\nu \bfz^\nu$ be an analytic function with $f({\bf 0})=0$.
Recall that the Newton diagram $\Ga_+(f)$ is the minimal convex set in the first quadrant of $\mathbb R_+^n$ containing
$\cup_{\nu,c_\nu\ne 0}(\nu+\mathbb R_+^n)$. %(We use $(\nu_1,\dots, \nu_n)$ be the canonical coordinates of this quadrant.) 
The Newton boundary $\Ga(f)$ is the union of compact faces of $\Ga_+(f)$.
Let $N_+$ be the space of non-negative weight vectors. Using the canonical basis, it can be identified with 
the first quadrant of $\mathbb R_+^n$. Let  $P=(p_1,\dots,p_n)\in N_+$ be a non-zero weight. It defines a canonical linear mapping on $\Ga_+(f)$ by
$P(\nu):=\sum_{i=1}^n p_i\nu_i$. We denote the minimal value of $P$ on $\Ga_+(f)$  by $d(P,f)$ and put 
$\De(P,f)=\{\nu\in \Ga_+(f)\,|\, P(\nu)=d(P,f)\}$. This is a face of $\Ga_+(f) $. We simply write as $d(P)$ or $\Delta(P)$ if no confusion is likely.
The dimension of $\De(P)$ can be $0,1,\dots, n-1$. Recall that 
$P,Q\in N_+$ are {\em equivalent} if $\De(P)=\De(Q)$. This equivalence classes gives $N_+$ a rational polyhedral cone subdivision $\Ga^*(f)$ and we call
 $\Ga^*(f)$ {\em the dual Newton diagram}. 
We also define a partial  order  in $\Ga^*(f)$ by
\[P\le Q\iff \De(P)\subset \De(Q).
\]
Denote the set of weights which are equivalent to  $Q$ by $[Q]$.
Note that the closure $\overline{[P]}$ is equal to the union  $\cup_{Q\ge P}[Q]$ and 
%Note that 
\[\dim\,\overline{[P]}=\dim\,[P]=n-\dim\,\De(P).
\]
The generators of the 1-dimensional cones of $\Ga^*(f)$ are called {\em vertices}. 
A weight  $P\in N_+$ is a vertex if and only if $\dim \De(P)=n-1$. We denote the set of vertices by $\mathcal V$.
 $P=(p_1,\dots, p_n)$ is called {\em strictly positive} if $p_i>0$ for any $i=1,\dots,n$.
Note that $\De(P)$ is compact if and only if $P$ is strictly positive. A vertex which is not strictly positive is either
the canonical basis
$e_i=(0,\dots,\overset {\overset{i}\smile}1,\dots, 1),\,1\le i\le n$ or corresponds to a vanishing coordinate subspace (see  \S \ref{Lojasiewicz non-degeneracy} 
for the definition).
\subsubsection{Face function}
For $\De\subset \Ga(f)$, put $f_\De(\bfz):=\sum_{\nu\in \De} c_\nu \bfz^\nu$ and we call $f_\De$ {\em the face function of $\De$}. For a weight vector $P\in N_+$, the face function associated with $P$ is defined by $f_P(\bfz)=f_{\De(P)}(\bfz)$. The monomial $\bfz^\nu$ and the integer point $\nu\in \Ga_+(f)$ correspond each other. If $\Delta$ is a compact face, $f_\Delta$ is a weighted homogeneous  polynomial.
 %$\nu\in \Ga(f)$ is called {\em the  corresponding exponent point} of $\Ga(f)$ to the 　　monomial $\bfz^\nu$.
%%%%%%%%%%%%%%%%%%%%
\subsection{Normalized weight vector}
Take a weight vector $P=(p_1,\dots, p_n)$. Let $d=d(P)$  and assume that $d>0$. 
%Then the face function $f_P$ is defined by 
%$f_P(\bfz)=\sum_{P(\nu)=d}c_\nu\bfz^\nu$. 
The hyperplane $\Pi$ in $\mathbb R^n$, defined by
$p_1\nu_1+\dots+p_n\nu_n=d$, contains  the face $\De(P)$ and all other points $\nu\in\Ga_+(f)\setminus\De(P)$ are above $\Pi$. Namely
$p_1\nu_1+\dots+p_n\nu_n>d$.
We call $\Pi$ {\em  the supporting hyperplane }of the weight vector $P$. For a weight vector $P$ with $d(P)>0$,
we define {\em the normalized weight vector} $\hat P$ of $P$ (with respect to $f$)  by
\[
\hat P:=(\hat p_1,\dots, \hat p_n),\quad \hat p_i=p_i/d(P).
\]
Herefater we use this notation $\hat P$ throughout this paper.  Using the normalized weight vector,  $d(\hat P)=1$ and the supporting hyperplane $\Pi$ is written as
\[\Pi:\quad \hat p_1\nu_1+\dots+ \hat p_n\nu_n=1.\]
Note that the $\nu_j$ coordinate of the intersection of $\Pi$ and $\nu_j$-axis is $1/{\hat p_j}$.

Assume that $\De(P)\cap \De(Q)\ne \emptyset$ and consider the line segment
$P_t:=tP+(1-t)Q,\,0\le t\le 1$. Note that $\De(P_t)=\De(P)\cap \De(Q)$ for any $0<t<1$
and the normalized weight vector $\hat P_t$ is simply given by
$\hat P_t=t\hat P+(1-t)\hat Q$, provided $d(P)>0$ and $d(Q)>0$.
%%%%%%%%%
%\subsubsection{Purpose}

The purpose of this paper is to give an upper bound explicitly for the \L ojasiewicz exponent using the combinatorial data of the Newton boundary.
 Then we give an application for the characterization of the monomials which  do not change  the topology by adding to $f$. For a weighted homogeneous non-degenerate polynomial,  we prove the estimation conjectured in \cite{B-K-O} under the \L ojasiewcz non-degeneracy.
 In \S 4,  we generalize these results for mixed functions.
%\end{comment}

%%%%%%%%%%%%%%%%%%
\section{\L ojasiewicz exponent for convenient functions}
%%%%%%%%%%%%%%%%
\subsection{Preliminary consideration}\label{Preliminary}
We first consider the estimation of
\nl  \L ojasiewicz exponent along an analytic curve $C(t)$ which is parametrized as 
follows.
Put $I:=\{i\,|\, z_i(t)\not\equiv 0\}$ and $I^c$ be the complement of $I$.
%%%%%%%%%%
\begin{eqnarray}\label{test-curve}
C(t):  \begin{cases}
&\bfz(t)=(z_1(t),\dots, z_n(t)),\,\,\bfz(0)=0,\, \bfz(t)\in \mathbb C^{*I}\\
&z_i(t)=a_it^{p_i}+\text{(higher terms)},\,p_i\in \mathbb N,\, i\in I\end{cases}
\end{eqnarray}
Here we use the following  notations:
\[\begin{split}
\mathbb C^{I}:=&\{\bfz=(z_1,\dots, z_n)\,|\, z_j=0,\,j\not\in I\}\\
\mathbb C^{*I}:=&\{\bfz=(z_1,\dots, z_n)\,|\, z_i\ne 0,\iff i\in I\}\\
N_+^I:=&\{P=(p_1,\dots,p_n)\in N_+\,|\, p_j= 0,\,j\not\in I\}\\
N_+^{*I}:=&\{P=(p_1,\dots,p_n)\in N_+\,|\, p_i\ne 0\iff i\in I\}\\
f^I:=&f|_{\mathbb C^I}.
\end{split}
\]
%Assume that $f(\bfz)$ is convenient. % and put $f^I:=f|_{\mathbb C^I}$.
Put  $P=(p_i)_{i\in I}\in N_+^{*I}$ and  $ d=d(P,f^I)$.
We define %$J:=\{j\,|\, \frac{\partial f_P}{\partial z_j}(\bfa)\ne 0\}$, 
\[M(P):=\max\{p_j\,|\, j\in I\},\quad m(P):=\min \{p_j\,|\, j\in I\}.
\]
%We assume that $f^I:=f|_{\mathbb C^I}$ is not identically zero.
Note that ${{\ord}}\,\bfz(t)= {m(P)}$. Put $q:={\ord}\,\partial f^I(\bfz(t))$.
Under the non-degeneracy assumption, %$J\ne \emptyset$.
 we have the equalities:
\begin{eqnarray}\label{estimation1}
 d-M(P)&\le&q\le d-m(P)\,\,\text{or}\\
\frac{d-M(P)}{m(P)}&\le& \frac q{m(P)}
\le \frac{d}{m(P)}-1=\frac 1{m(\hat P)}-1.
\end{eqnarray}
Put $Vari(P)=\{z_j\,|\, \frac{\partial f_P}{\partial z_j}\not \equiv 0 \}$. Namely ${Vari}(P)$ is the set of variables which appear in
$f_P$.
Then we have the obvious estimations:
\begin{eqnarray}
%\begin{split}
&&\frac{\partial f}{\partial z_j}(\bfz(t))=\left(\frac{\partial f}{\partial z_j}\right )_P(\bfa)t^{d_j}+\text{(higher terms)},\, d_j=d(P,\frac{\partial f}{\partial z_j}),\\
&&{\ord}\,\frac{\partial f}{\partial z_j}(\bfz(t))\ge  d_j\ge d-p_j.
\end{eqnarray}
If $z_j\in {Vari}(P)$, $d_j=d(P,f)-p_j$ and otherwise $d_j > d-p_j$.
If $m(P)=p_j$, $ d/{m(P)}={1}/{{\hat p}_{j}}$ and 
this  is  equal to the $j$-th coordinate of the intersection of $\Pi(P)$ and $\nu_j$ axis.
We define {\em the \L ojasiewicz exponent of $f$ along $C(t)$} by
\[
\ell_0(C(t)):=\frac{{\ord}\,\partial f(\mathbf z(t))}{{\ord}\,\mathbf z(t)}.
\]
By (\ref{estimation1}) and by the non-degeneray assumption, we have
\begin{eqnarray}\label{estimation2}
{\ord}\, \partial f(\mathbf z(t))&\le& d-m(P)\\
{\ord}\, f_j(\mathbf z(t))&\ge& d(P,f_j)\ge d-p_j\\
\ell_0(C(t))&\le& \frac{d-m(P)}{m(P)}.
\end{eqnarray}
For a strictly positive weight vector $P=(p_1,\dots, p_n)\label{invariants}$, we define positive invariants
\begin{eqnarray}\label{invariants}
\begin{cases}
&\eta_{i,j}(P):=\dfrac{d-p_j}{p_{i}}=\dfrac{1-\hat p_j}{\hat p_{i}},\, \, \\
&\eta_{i,j}'(P):=\dfrac{d_j}{p_i}=\dfrac{\hat d_j}{\hat p_i}\\
 &\eta(P):=\dfrac{d-m(P)}{m(P)}=\dfrac 1{m(\hat P)}-1.
 \end{cases}
\end{eqnarray}
where $d_j=d(P,f_j)$ and $\hat d_j=d_j/d$.
%where we assume that $p_\iota=p_{min}$.
%%%%%%%%%%%%
\subsection{\L ojasiewicz exceptional monomial}
We say that
$f(\bfz)$  is {\em convenient} if for any $1\le j\le n$, Newton boundary $\Ga(f)$ intersects with 
$\nu_j$-axis  at a point $B_j=(0,\dots, \overset{\overset j\smile} b_j,\dots, 0)$.
%We say $z_j^{b_j}$ a $j$-{\em axis monomial}.
Recall that a non-degenerate function $f(\bfz)$ has an isolated singularity at the origin, if it is convenient (\cite{Okabook}).
Assume that $f(\bfz)$ is convenient as above.
Define an integer $B:=\max\{b_j\,|\, j=1,\dots, n\}$ and let  $\mathcal L=\{i\,|\, b_i=B\}$. We call
$z_i^B, \, i\in \mathcal L$ a {\em\L ojasiewicz monomial} of $f$.
%\subsection{Lojasiewich exceptional}
We say that a \L ojasiewicz monomial $z_i^B$ is {\em \L ojasiewicz exceptional} if  there exists $j,\,j\ne i$  and a monomial of the form
$z_jz_i^{B'}$,  with $B'<B-1$ which has a non-zero coefficient in $f$.

Consider the curve parametrized as (\ref{test-curve}) and assume that $I=\{1,\dots,n\}$.
Then we have seen 
\[
 \frac{{\ord}\,\partial f(\bfz(t))}{{\ord}\,\bfz(t)}\le \frac {d}{m(P)}-1\le B-1. 
\]
This implies  the following inequality holds in a small neighbourhood of the origin.
\begin{eqnarray}\label{ineq-conv}
\|\partial f(\bfz(t)\|&\ge c \|\bfz(t)\|^{B-1},\,\,c\ne 0.
\end{eqnarray}
If $z_i^B$ is \L ojasiewicz exceptinal and let $z_jz_i^{B'}, B'<B-1$  be as above.
Then $\frac{\partial f}{\partial z_j}$ has the monomial  $z_i^{B'}$ with non-zero coefficient and ${\ord}\, \frac{\partial f}{\partial z_j}(\bfz(t))$ can be
$p_iB'$ which is smaller than $p_i(B-1)$. In fact,  this is the case for the $i$-axis  curve $\bfz(t)$
where $z_i(t)=t$ and $z_j(t)\equiv 0$ for $ j\ne i$.
We assert
\begin{Assertion} The inequality $\ell_0(f)\le B-1$ holds for any analytic curve $\bfz(t)$.
\end{Assertion}
\begin{proof}As we have shown the assertion for the case $I=\{1,\dots,n\}$, we
need only  consider the case where  some of $z_i(t)$ is identically zero. In this case, put $I:=\{i\,|\, z_i(t)\not\equiv 0\}$.
% and $J$ be the complement of $I$.
%For simplicity, we assume that $J=\{1,\dots, m\}$. 
Then $f^I:=f|_{\mathbb C^I}$ is a non-degenerate convenient function.
Thus by the above argument applied for $f^I$, we have
\[
{{\ord}}\, \partial f^I(\bfz(t))\le ({{\ord}}\,{\bfz}(t))\times ({B_I-1}).
\]
Here $B_I$ is defined similarly for $f^I$.
By the obvious inequality ${\ord}\, \partial f(\bfz(t))\le{\ord}\,\partial f^I(\bfz(t))$ and $B_I\le B$,
we get the inequality (\ref{ineq-conv}).
\end{proof}

For the practical calculation of the \L ojasiewicz exponent, we use the following criterian.
 This can be proved by 
 the Curve Selection Lemma (\cite{Milnor,Hamm1}).
 \begin{Proposition}\label{CurveSelection} A positive number 
$\theta$  satisfies the \L ojasiewicz inequality (\ref{Loj-ineq}) if the inequality
\[{{\ord}}\, \partial f(\bfz(t))\le \theta\times {\ord}\, \bfz(t)
\]
 is satisfied
along
any non-constant analytic curve $C(t)$ parametrized by an analytic path $\bfz(t)$ with $\bfz(0)={\bf 0}$. That is
$\ell_0(f)=\sup\, \ell_0(C(t))$ where $C(t)$ moves every possible analytic curves starting from the origin.
\end{Proposition}
Now we have the following  result for convenient non-degenerate functions.%Thus we conclude
\begin{Theorem}
\label{convenientLojasiewicz}
 Let $f(\bfz)$ % (respectively $f(\bfz,\bar\bfz)$)
 be a  non-degenerate convenient analytic function.
%(resp. mixed analytic function) and let $M$ be as above.
Then \L ojasiewicz exponent   $\ell_0(f)$ satisfies the inequality:
$\ell_0(f)\le B-1$.

Furthermore if $f$ has a \L ojasiewicz non-exceptional monomial, $\ell_0(f)=B-1$.
\end{Theorem}
\begin{proof}
We have shown that $\ell_0(f)\le B-1$. We only need to show the existence of a curve $C(t)$ which takes the equality $q=B-1$,
assuming that  $f$ has a \L ojasiewicz non-exceptional monomial.
For this purpose,  we assume for simplicity $B=b_1$ and $z_1^B$ is non-exceptional. Note that 
the Newton boundary of $\frac{\partial f}{\partial z_i}$ does not touch the $\nu_1$ axis under $B-1$ for any $i>1$ by the assumption. Thus we can take a sufficiently large integer
$N$ and put $P=(1,N,\dots, N)$.
Note that $d(P,\frac{\partial f}{\partial z_1})=B-1$ and 
$d(P,\frac{\partial f}{\partial z_i})\ge B-1$ for any $i\ge 2$.
Consider the curve $C(t)$ defined by $\bfz(t)=(t,t^N,\dots, t^N)$.
Then the above observation tells us that 
\[
\partial f(\bfz(t))=(B,*,\dots, *)t^{B-1}+\text{(higher terms)}.
\]
Thus $\|\partial f(\bfz(t))\|\approx \|\bfz(t)\|^{B-1}$.
\end{proof}
%\begin{Remark} 
In the above proof, if there is a monomial $z_1^{B'}z_j$ with $j\ne 1, B'<B-1$, we see that ${{\ord}}\,f_j(\bfz(t))= B'$.
% If the equality holds,
%we have $B-1={\ord}\,f_1(bfz(t))>{\ord}\,f_j(\bfz(t))$.
%
Thus we have ${\ord}\,f_1(\bfz(t))>{\ord}\,f_j(\bfz(t))$.
The importance  of \L ojasiewicz exceptional monomial is observed by Lemarcik \cite{Len}. For  plane curves ($n=2$), we have also observed that it gives a fake effect 
to computation of  the complexity of plane curve singularity  but   exceptional monomials can be eliminated without changing the non-degeneracy
 %in the process of computing the complexity of plane curves
(Le-Oka \cite{Le-Oka}). %In fact, %we have also considered such face in \cite{Le-Oka}
Suppose that $c\, z_1^{B}+c' z_1^{B'}z_2$ with $B'\le B-2, c,c'\ne 0$ is in  a face function of $f$.
Then take the coordinate change $(z_1,z_2'):=(z_1,z_2+ (c/c')z_1^{B-B'})$ to kill the monomial 
$z_1^{B}$. This operation does not work for mixed polynomials.
%\end{Remark}
%%%%%%%%%%%%%%%%%%%
\section{\L ojasiewicz exponents for non-convenient functions}
In this section, we consider again a non-degenerate  function $f(\bfz)$ with isolated singularity at the origin without  assuming  the convenience of the Newton boundary. It turns out that the estimation of \L ojasiewicz exponent is much more complicated without the convenience assumption.
We assume that $\Gamma(f)$ has dimension $n-1$ herefater. If the multiplicity at the origin is greater than 2, this condition is always satisfied.
%or the simplicity, we assume that $\dim\,\Ga(f)=n-1$.
%$f(\bfz,\bar\bfz)$ being non-convenient, non-degenerate function.　We first consider the holomorphic functions.
%%%%%%%%%%%%%%%%%%%%%%%%%%%%%%%%%
\subsection{\L ojasiewicz non-degeneracy along a vanishing coordinate subspace}\label{Lojasiewicz non-degeneracy}\label{Lojasiewicz non-degeneracy}
Let $I$ be a subspace of $\{1,\dots, n\}$. We say that $\mathbb C^I$ is a 
{\em a vanishing coordinate subspace}  (\cite{OkaMix,OkaAf,EO13}) if $f^I(\bfz_I)\equiv 0$.  Here $f^I$ is the restriction of $f$ to $\mathbb C^I$.
%$f({\bfz_I},{\bf 0}_J)\equiv 0$ where $J=I^c=\{m+1,\dots, n\}$
We use the notation $\mathbb C^I=\{\bfz\,|\, z_{j}=0,\,j\not\in I\}$ and 
$\bfz_I=(z_i)_{i\in I}$. 
If further  $I=\{i\}$ is a vanishing coordinate subspace, we say $\mathbb C^{\{i\}}$ {\em a vanishing axis}. 
We say a  face $\Xi\subset \Ga_+(f)$ is {\em essentially non-compact} if there exists a
 non-strictly positive  weight function $Q=(q_1,\dots, q_n)$
such that $d(Q,f)>0$ and $\De(Q)=\Xi$. Let $I(Q)=\{i\,|\, q_i=0\}$.
% and put $J=I^c$. 
We say also  $I(Q)$ {\em the vanishing direction of $\Xi$} and write also  as $I(\Xi)=I(Q)$.
Then the assumption $d(Q,f)>0$ implies $\mathbb C^I$ is a vanishing coordinate subspace.
%%%%%%
%%%%%%%%%%%

Put $I=I(Q)$ and 
we assume that $I=\{1,\dots, m\}$.
Take  an $i\in I$. By the assumption, the gradient vector $\partial f$ does non vanish 
in a neighbourhood of the origin of $i$-axis except at the origin. This is possible only if 
there  exists a monomial $z_i^{n_i} z_j$ with a non-zero coefficient for some $j\ne i$
in the expansion of $f$. Then we observe that $j\not \in I$, because  $\mathbb C^I$ is  a vanishing coordinate subspace. Let $J_i$ be the set of $j\in I^c$ for which such a monomial $z_i^{n_i} z_j$ exists with a non-zero coefficient in the expansion of $f(\bfz)$.
% (To see this, it is enough to observe that $\{i\}$ is a vanishing coordinate subspace and consider the gradient vector on this axis.)
$J_i\ne \emptyset$ for any $i\in I$. We define an integer $n_{ij}$ by
\[n_{ij}:=\min\{n_i\,|\, z_i^{n_i}z_j~\text{has a non-zero coefficient}\}
\] 
for a fixed $i\in I$ and $j\in J_i$. %We call $z_i^{n_{ij}}z_j$ {\em an almost $i$-axis monomial.}
 For brevity, we put $n_{ij}=\infty$ if $j\notin J_i$.
Put
$B_{ij}:=(0,\dots, \overset {\overset i \smile} {n_{ij}},\dots, \overset {\overset j \smile}{1},\dots,0)$. Note that $B_{ij}\in\Ga(f)$.
 Put $J(I)=\cup_{i\in I} J_i$. % Assume that $I=\{1,\dots, m\}$ %and $J(I)=\{j_1,\dots, j_k\}$ 
%for simplicity.
We say that $f$ is {\em  \L ojasiewicz non-degengerate} if for  any strictly positive  weight vector $P\in N^{*I}$,
the following condition is  satisfied. Put 
$I':=\{i\,|\, p_i=m(P)\}$ and $J(P):=\cup_{i\in I'}J_i\subset J(I)$.
Then
the variety
\[\{\bfz\in \mathbb C^{*I}\,|\, \,\left(\left(f_j\right)^I\right)_{P}(\bfz_I)=0,\, \forall j\in J(P)\}.
\]
is empty. In other word, for any $\bfa\in \mathbb C^{*I}$, there exists $j\in J(P)$ such that $((f_j)^I)_{P}(\bfa)\ne 0$.
Here and hereafter we use the simplified  notation for the derivative function: $f_i(\bfz):=\frac{\partial f}{\partial z_i}(\mathbf z),\,i=1,\dots, n$.
%%%%%%%%%%%%%%%%%%%%%
\subsection{Jacobian dual Newton diagram} We consider the derivatives $f_i(\bfz)$,
$i=1,\dots, n$.
We consider their Newton boundary $\Ga(f_i)$, $i=1,\dots, n$. As we consider $n+1$ Newton boundaries,
we denote by $\Delta(P,f_i)$ the face of $\Gamma(f_i)$ where $P$ takes minimal value, $d(P,f_i)$.
We consider the following stronger equivalence relation in the space 
of non-negative weight vectors.   Two weight vectors $P,Q$ are {\em Jacobian equivalent} if
$\Delta(P,f_i)=\Delta(Q,f_i)$ for any $i=1,\dots,n$ and $\Delta(P,f)=\Delta(Q,f)$. 
We denote it by $P\underset J{\sim} Q$. This gives a polyhedral cone subdivision of $N_+$ and we denote this as 
$\Gamma_{J}^*(f)$ and we call it {\em the Jacobian dual Newton diagram of $f$}. $\Gamma_J^*(f)$ is a polyhedral cone subdivision of $N_+$
which is finer than $\Gamma^*(f)$.
 \newline\indent
The Jacobian dual Newton diagram  can be understood alternatively as follows. Let us consider the function 
$F(\mathbf z)=f(\mathbf z)f_1(\mathbf z)\cdots f_n(\mathbf z)$. Then $\Gamma_J^*(f)$ is essentially equivalent to the dual Newton diagram $\Gamma^*(F)$
of $F$. For any weight vector $P$, we have $\Delta(P,F)=\Delta(P,f)+\Delta(P,f_1)+\dots+\Delta(P,f_n)$ where the sum is Minkowski sum.
See \cite{Bo-Fe} for the definition.
For a weight vector $P$, the set of equivalent weight vectors in $\Gamma^*(f)$ and $\Gamma_J^*(f)$ is denoted as 
$[P]$ and $[P]_J$ respectively.
We consider the vertices of this subdivision.
We denote the set of strictly positive vertices of $\Gamma^*(f)$ and $\Ga_J^*(f)$ by $\mathcal V^+,\,\mathcal V_J^+$ respectively.
Recall that   $e_i=(0,\dots,\overset{\overset i\smile}1,\dots,0)$.
\begin{Proposition}
\begin{enumerate}
\item $P\underset J\sim Q$ implies $P\sim Q$ in $\Gamma^*(f)$. Conversely if $P\sim Q$ and $f_P(\mathbf z)$ contains all $n$-variables, $P\underset J\sim Q$.
\item A strictly positive weight vector $P$ is in $\mathcal V^+$ or $\mathcal V_J^+$ if and only if 
$\dim\Delta(P,f)=n-1$ or $\dim\,(\Delta(P,f)+\sum_i\,\Delta(P,f_i))=n-1$ respectively where the summation is  Minkowski sum.
\end{enumerate}
\end{Proposition}
\noindent
%See \cite{Okabook, Bo-Fe} for the definition of Minkowski sum.
Let  $\mathcal V_0$ be the set of vertices of $\Gamma^*(f)$ which is not strictly positive. 
\begin{Proposition}
Assume that $P\in \mathcal V_0$ 
and $\mathbf C^{I(P)}$ is a non-vanishing subspace. Then $P$ is one of $e_1,\dots, e_n$.
\end{Proposition}
%A vertex $P\in\mathcal V_0$ is not important if $\mathbb C^{I(P)}$ is not vanishing subspace.
\begin{Lemma}Let $P=(p_1,p_2,\dots, p_n)$ be a non-elementary
vanishing  vertex of $\Gamma^*(f)$ in $\mathcal V_0$ and put $I=I(P)$.
Then the following holds.
\begin{enumerate}
\item $f_P$ contains every  variable $z_1,\dots, z_n$. In particular, $\hat p_i\le 1$ for any $i$.
\item  Any  monomial  $z_i^{a}z_j,\,i\in I$ must  be contained in $f_P(\mathbf z)$, as
$\hat p_i=0$ and  $\deg_{\hat P}z_i^{a}z_j=\hat p_j\le 1$.
\item There are no monomials $\mathbf z^\nu\in \mathbb C[\mathbf z_I]$  in $f_P(\mathbf z)$.
\end{enumerate}
\end{Lemma}
\begin{proof}
Suppose that $f_P$ does not contain the variable $z_i$ for some $i$. Then $\Delta(P)\subset \{\nu_i=0\}$. 
%As $d(P,f)>0$, therefore 
Then $e_i\in \overline{[P]}$ and thus a contradiction
$d(P,f)=d(e_i,f)=0$. This proves the first assertion.
Consider   $z_i^{a}z_j,\,i\in I $ in $f$. Consider the normalized vector $\hat P$.
Then $\hat p_i\le 1$ by the assertion (1) and  as  $\deg_{\hat P}z_i^az_j\ge 1$, this implies $\hat p_j=1$ and $z_i^{a}z_j$ must be in $f_P$.
If there is a monomial $\mathbf z^{\nu}$ as in the assertion, $\deg_{\hat P}\mathbf z^\nu=0$ and an obvious contradiction.
\end{proof}
%\newline\noindent
\begin{Example} \label{example-Jdual}
Consider $f(\mathbf z)=(z_1^9+z_2^3+z_3^6)z_2+z_3^7+z_4^7$.
$\mathcal V_J^+$ has  vertices $e_1,e_2,e_3,e_4$ and $R,P,S$ where
\[\begin{split}
&R=(\frac 1{12},\frac 14,\frac 17,\frac 17),\quad f_R(\mathbf z)=z_1^9z_2+z_2^4+z_3^7+z_4^7\\
&P=(\frac 2{21},\frac 27,\frac 17,\frac 17),\quad f_P(\mathbf z)=z_3^7+z_4^7\\
&S=(0,1,\frac 17,\frac 17),\quad f_S(\mathbf z)=z_1^9 z_2+z_3^7+z_4^7.
\end{split}
\]
Note that $P\in \mathcal V_J^+\setminus \mathcal V^+$ as  $f_{2P}=z_1^9+z_2^3+z_3^6$ and $\deg_P f_{2P}+\frac 27=\frac87 > 1$. $S\in \mathcal V_0$ corresponds to the vanishing coordinate subspace $\mathbb C^{\{1\}}$.
The vertex $P$ is in the simplicial cone $\Cone (R,e_1,e_2)$ as $P=R+\frac 1{84}e_1+\frac1{28}e_2$.
Note that $\Cone (R,e_1,e_2)$ is a regular boundary region.  See the definition below. \end{Example}
%%%%%%% %%
\subsection{ Boundary region}
We consider equivalence classes $[P]$ and $[P]_J$ in $\Gamma^*(f)$
and $\Gamma_J^*(f)$ respectively. There exist three different  cases.
\newline
\noindent
1. An equivalent class $[P]$ (respectively $[P]_J$ ) is called {\em an inner  region}
if the closure $\overline{[P]}$ (resp. $\overline{[P]_J}$ ) does not contains any  vertex 
of $\mathcal V_0$ on the boundary.
\newline\noindent
2. $[P]$ (respectively $[P]_J$) is called {\em a regular boundary region}
if the closure $\overline{[P]}$ (resp. $\overline{[P]_J}$)  contains some  vertex $e_i$ but contains no vanishing vertex on the boundary.
 \newline\noindent
3. $[P]$ (resp. $[P]_J$) is called {\em a vanishing boundary region}, if  $\overline{[P]}$ (resp. $\overline{[P]_J}$) contains a vanishing  vertex $ Q\in \Gamma^*(f)$ (resp. $ Q\in \Gamma_J^*(f)$) on the boundary.
 
\subsection{Special admissible paths} %Some structure of $\Gamma^*(f),\,\Gamma_J^*(f)$}
Two weight vectors $P,Q$
are called {\em admissible}, (respectively {\em $J$-admissible}) if
$\De(P)\cap\De(Q)\ne \emptyset$ (resp. $\De(P)\cap\De(Q)\ne \emptyset$ and $\De(P,f_i)\cap \De(Q,f_i)\ne \emptyset$ for any $i$).
Any weight $R$ in the interior of an admissible lene segment $\overline{PQ}$ satisfies
$\De(R)=\De(P)\cap \De(Q)$ (resp. $\De(R)=\De(P)\cap \De(Q)$ and $\De(R,f_i)=\De(P,f_i)\cap \De(Q,f_i),\, i=1,\dots, n$).
 Take a weight vector $P=(p_1,\dots, p_n)$ which is  not strictly positive.
Put $I=\{i\,|\, p_i=0\}$. We say $P$ is a {\em vanishing weight } (respectively {\em non-vanishing weight}) if $f^I\equiv 0$ (resp. $f^I\not \equiv 0$).
\begin{Proposition}\label{non-vanishing-segment} (A path in a regular boundary region) Suppose that two weight  ${P, Q}$ are admissible, $Q$ is strictly positive weight and $P$ is a  non-vanishing weight vector
with $I=\{i|p_i= 0\}$.
Then weight vector $R\in \overline{PQ}$ on this line segment  (except $P$) is given in the normalized form as 
$\hat R_t=\hat Q+t P$ with $0\le t <\infty$. In this expression, $\hat R_{t}\to P$ when $t\to \infty$ and there exists a sufficiently large $\delta>0$
so that $m(\hat R_t)\equiv m(\hat Q_I)$ and $\eta(R_t)\equiv\eta(\hat Q_I)$ for $t\ge \delta$ and 
$\eta(\hat Q_I)\le \eta(\hat Q)$. 
\end{Proposition}
\begin{proof}For $j\not\in I$, $\hat q_j+tp_j\to \infty$ and  the assertion follows immediately.
Here $m(\hat Q_I)=\min\{\hat q_j|j\in I\}$.
\end{proof}%
\begin{Proposition} (A path in a vanishing boundary region)
Suppose that $Q$ is strictly positive and  $P$ is a vanishing weight vector. Put $I=\{i|p_i=0\}$. Then $\hat Q_t=(1-t)\hat Q+t\hat P,\,(0\le s\le 1)$
parametrize the weights on the line segment $\overline{QP}$ and 
  we have the following.
\begin{enumerate}
%\item There exists a neighborhood $U$ of $P$ in $N_+$ so that  for any strictly positive  $R=(r_1,\dots, r_n)\in U$, there exists some $i\in I$ and $m(R)$ is taken by 
%${r_i}$.
\item Suppose that  $\hat q_j\ge \hat q_i$ for some $i\in I,\,j\not\in I$. Then  $(1-t)\hat q_j+t \hat p_j\ge (1-t) \hat q_i$ for $0\le t\le 1$.
\item If there is a $j\not \in I$ such that   $\hat q_j< \hat q_i$ for some $i\in I$, there exists $0<t_0<1$ which satisfies  
$(1-t_0)\hat q_j+t_0 \hat p_j=(1-t_0) \hat q_i$.
\end{enumerate}
\end{Proposition}
\begin{proof}Third assertion follows from the following property.
\[(1-t)\hat q_j+t \hat p_j\underset{t\to 1}\to \hat p_j>0,\,(1-t)\hat q_i\underset{t\to 1}\to 0.
\]
\end{proof}
%%%%%%%%%%%%%%
\subsection{Key lemma}
First we prepare an elementary lemma.
\begin{Lemma} \label{monotone1}
Consider a linear fractional function
$\varphi(s)=\frac{as+b}{cs+d}$ where $a,b,c,d$ are real numbers such that $(c,d)\ne (0,0)$ and $cs+d\ne 0$ for $0\le s\le 1$.
Then if $\varphi'(s)\not\equiv 0$, the sign of $\varphi'(s)$ does not change i.e.,
  $\varphi'(s)>0$ or $\varphi'(s)<0$  for any $s,\,0\le s\le 1$. Thus $\varphi(s)$ is a monotone function %(including a constant function)
   on $[0,1]$.
\end{Lemma}
\begin{proof}Assertion follows from
\[\vphi'(s)=\frac{ad-bc}{(cs+d)^2}.
\]
\end{proof}
 Assume that $P,Q$ are   strictly positive weight vectors.
Then the weights on this line segment $\overline{PQ}$
  can be parametrized  normally as
$\hat R_s,\,0<s<1$:
%we have the following equality.
 \[
\hat R_s=s\hat P+(1-s)\hat Q.
\]
%\end{Lemma}
Putting $\hat P=(\hat p_1,\dots, \hat p_n)$ and $\hat Q=(\hat q_1,\dots, \hat q_n)$,
we can write  $\hat R_s=(s\hat p_1+(1-s)\hat q_1,\dots, s\hat p_n+(1-s)\hat q_n)$.
%\subsubsection{Strictly positive case}
We consider the quantities defined in (\ref{invariants}):
\[\begin{split}
&\eta_{ij}(\hat R_s)=\frac{1-\hat r_{s,j}}{\hat r_{s,i}}=\frac {1-(s\hat p_j+(1-s)\hat q_j)}{ s\hat p_i+(1-s)\hat q_i},\,\\
&\eta(R_s)=\frac{1-m(\hat R_s)}{m(\hat R_s)}\\
&\eta_{ij}'(R_s)=\frac{\deg(\hat R_s,f_j)}{\hat r_{s,i}}=\frac {\deg(\hat R_s,f_j)}{ s\hat p_i+(1-s)\hat q_i}
\end{split}
\]
Applying Lemma\,\ref{monotone1}, we have
\begin{Lemma}\label{monotone2}
 Assume that $P,Q$ are   strictly positive weight vectors.
\begin{enumerate}
\item 
Assume that $P,Q$ are addmissible. Then we have
\[
 \eta_{ij}(\hat R_s)\le \max\, \{\eta_{ij}(\hat P),\eta_{ij}(\hat Q)\},\, 0<s<1.
\]
In particular, we  have
\[
 \eta(\hat R_s)  \le \max\{\eta(\hat P),\eta(\hat Q)\}
\]
\item Assume that $P,Q$ are $J$-addmissible. Then we have
\[
 \eta_{ij}'(\hat R_s)\le \max\, \{\eta_{ij}'(\hat P),\eta_{ij}'(\hat Q)\},\, 0<s<1.
\]
\end{enumerate}
\end{Lemma}
%%%%%%%%%%%%%%%%%%%%
\subsubsection{Invariants to be used for the estimation}
Let $\mathcal V_J^+$ be the set of strictly positive vertices of $\Gamma_J^*(f)$  and consider the subset  $\mathcal V_J^{++}\subset \mathcal V_J^+$ which are in a vanishing boundary region ${[Q]}$ of $\Gamma^*(f)$ for some $Q$.
The numbers of $\mathcal V^+, \,\mathcal V_J^+,\,\mathcal V_J^{++}$ are finite. We define the following invariants.
\[\begin{split}
&\eta_{max}(f):=\max\{\eta(P)\,|\, P\in \mathcal V^+\}\\
&\eta_{J,max}(f):=\max\{ \eta(P)\,|\, P\in \mathcal V^+\cup \mathcal V_J^{++}\},\\
&\eta_{J,max}'(f):=\max\{\eta_{k,i}'(R)\,|\,R\in  \mathcal V_J^{++},\,k,i=1,\dots,n\}\\
&\eta''_{J,max}(f):=\max\{\eta_{J,max}(f),\eta'_{J,max}(f)\}.
\end{split}\]
Here $\eta_{k,i}'(R)={d(R,f_i)}/{r_k}$.
%%%%%%%%%%%%%%%%%%
\subsection{Main theorem}\label{MainTheorem2}
The following estimation is our main result which is  a modified weaker version of the  assertion in \cite{Fukui}. Recall that we assume that $\dim\, \Ga(f)=n-1$.
\begin{Theorem}\label{main-theorem}
Let $f(\bfz)$ be a  non-degenerate \L ojasiewicz non-degenerate function with an isolated singularity at the origin.
Then \L ojasiewicz exponent $\ell_0(f)$ has the estimation
\[\ell_0(f)\le \eta_{J,max}''(f).\]
If $\mathcal V_J^{++}=\emptyset$, the estimation can be replaced by a better one
\[\ell_0(f)\le \eta_{max}(f).\]
\end{Theorem}

The proof of Theorem \ref{MainTheorem2} will be given in \S \ref{ProofOfMainTheorem}.
%%%%%%%%%%%%%%
%%%%%%%%%%%%%%%

\subsubsection{Test Curve}
We consider an analytic curve $C(t),\,-\eps<t\le \eps$ parametrized as before (\ref{test-curve}),
%Put $I:=\{i\,|\, z_i(t)\not\equiv 0\}$ and $I^c$ be the complement of $I$.
\begin{eqnarray}\label{test-curve2}
C(t):  \begin{cases}
&\bfz(t)=(z_1(t),\dots, z_n(t)),\,\,\bfz(0)=0,\, \bfz(t)\in \mathbb C^{*I}\\
&z_i(t)=a_it^{p_i}+\text{(higher terms)},\quad i\in I\\
&z_j(t)\equiv 0,\,\,j\not\in I
\end{cases}
\end{eqnarray}
For simplicity, we assume that $I=\{1,\dots, m\}$.
Put $P=(p_1,\dots, p_m)\in N_+^{*I}$ and $\bfa=(a_1,\dots, a_m)\in\mathbb C^{*I}$.
%Let $\mathcal P:=(w_1,\dots, w_n)$ be the normalization of $P$.
We are interested in  a best possible upper bound for the positive quantity 
$\ell_0(C(t))=
{\ord}\, \partial f(\bfz(t))/{\ord}\,\bfz(t)$.
Using the notation $f_j=\partial f/\partial z_j$,  
we have  the expansion
 %for $i\in I$:
\begin{eqnarray}
f_j(\bfz(t))&=&
                 ( f_j)_P(\bfa) t^{d(P,f_j)}+\text{(higher terms)}. %\,j\in I^c.
 \end{eqnarray}
Note that if $f_P(\bfz)$ contains the variable  $z_j$, 
 \begin{eqnarray} 
 ( f_j)_P(\bfa)= (f_P)_j(\bfa) ,\quad {d(P,f_j)}={d(P,f)-p_j}.\end{eqnarray}
%f_j(\bfz(t))&=&( (f_j)^I)_P(\bfa) t^{d(P,(f_j)^I)}+\text{(higher terms)},\, \quad j\in J\\
\subsubsection{Curves corresponding to a  strictly
 positive weight vector} In the previous section, we have seen that a test curve $C(t)$ gives a pair $(P,\mathbf a)\in N_+^I\times \mathbf C^{*I}$.
 We consider the converse in the case $I=\{1,\dots, n\}$.
Assume we have a strictly positive integer weight vector $P=(p_1,\dots, p_n)\in N_{+}^n\cap \mathbb Z^n$. 
Taking a  coefficient vector $\bfa=(a_1,\dots, a_n)\in\mathbb C^{*n}$,
we associate an analytic curve
\[
C_P(t,\bfa):\, \bfz(t):=(a_1t^{p_1},\dots, a_nt^{p_n}).
\]
%and we define 
%\[\ell_0(C_P(t,\bfa)):={\ord}\,\partial f({\bfz}(t))/{\ord}\, {\bfz}(t).\]
The test curve (\ref{test-curve2}) gives the data $(P,\mathbf a)$ and if $I=\{1,\dots, n\}$, $C_P(t,\bfa)$ and $C(t)$ differs only higher terms.
%The estimation we give later depends only on $P$. Therefore
In this case, we also use the notation as
$\ell_0(P)$  instead of $\ell_0(C_P(t,\mathbf a))$ or $\ell_0(C(t))$ by an abuse of notation.
 %Let $p_{i_0}=m(P)=\min\,\{p_1,\dots, p_n\}$.
Then by the above discussion and by the non-degeneracy assumption, we have
\begin{eqnarray}\label{estmation2}
{\ord}\,\partial f(\bfz(t))&\le& d(P,f)-m(P)\\
\frac{{\ord}\,\partial f(\bfz(t))}{{\ord}\,\bfz(t)}&\le &\frac{d(P,f)}{m(P)}-1\\
&=& \frac 1{m(\hat P)}-1.
%&\eta(\hat P)&=\eta(P).
\end{eqnarray}
Note that $\eta(\hat P)=\eta(P)$.
This estimation does not depend on the choice of representative of
the equivalence class of $P$ and the choice of $\bfa$. 
The weakness of the above estimation is that $m(\hat P)$ can be arbitrary small in the vanishing boundary region which makes $\eta(\hat P)=\eta(P)$ unbounded.
% in the vanishing zone.
%%%%%%%%%%%%
\subsection{Proof of Theorem \ref{main-theorem}}\label{ProofOfMainTheorem}
Take a test curve as (\ref{test-curve}) and we consider the weight vector $P$. 
To prove the theorem, it is enough to prove that
$\ell_0(f)(C(t))\le \eta_{J,max}''(f)$ by Proposotion \ref{CurveSelection}.
We first consider  the  case that $P$ is either strictly positive which is most essential.
%The case $J$ is non-empty, the estimation has been reduced to the cases $J=\emptyset$ in the previous arguments.
\subsubsection{Strictly positive case}
We first   assume that $I=\{1,\dots, n\}$  and $P$ is strictly positive.
We divide the situation into  three cases.
\begin{enumerate}
\item[C-1] ${[P]}$ is an inner region. That is, $\overline{[P]}$  has only strictly positive weight vectors in the boundary.
\item[C-2] ${[P]}$ is  a regular boundary region.
% or a vanishing boundary region but ${[P]_J}$ is an inner region. % in $\mathcal V_J^+$.
\item[C-3] 
$[P]$ is a vanishing boundary region. In this case, we need to consider
the subdivision by Jacobian dual Newton diagram. There are three subcases.
\newline
C-3-1. ${[P]_J}$ is  an inner region.
%regular boundary region.
\newline C-3-2.
${[P]_J}$ is  a regular boundary region.
\newline
C-3-3. ${[P]_J}$ is also a vanishing boundary region.
\end{enumerate}
We need consider the Jacobian dual diagram only in  vanishing boundary regions of $\Gamma^*(f)$. %\newline\noindent
\subsubsection{ Cases  C-1 and C-3-1}
We start from the inequality
$\ell_0(P)\le \eta(P)$ and then estimate $\eta(P)$ by the strictly positive vertices.
We use an induction on $\dim\,[P]$ or $\dim\, [P]_J$ in Case C-3-1 to show that $\ell_0(P)\le \eta_{max}(f)$(resp. $\ell_0(P)\le \eta_{J,max}(f)$). %As the argument is similar, we explain  Case C-1.
The induction starts from the case $\dim\,[P]=1$ (respectively $\dim\,[P]_J=1$). 
In this case,  the assertion is obvious.
If $\dim\, [P]=r>1$, we take a line segment $\overline{RS}$
with $R,S\in \partial \overline{[P]}$ (resp. in $\overline{[P]_J}$) passing through $P$ we apply Lemma\,\ref{monotone2}
to get the estimation
\[
\begin{split}
&\ell_0(C(t))\le \eta(P)\le\max\{\eta(R),\eta(S)\}\le \eta_{max}(f),\, R,S\,:\text{admissible}\\
&\ell_0(C(t))\le  \eta(P)\le\max\{\eta(R),\eta(S)\}\le \eta_{J,max}(f), \,S\,:\text{$J$-admissible.}
\end{split}
\]
As $\dim\,[R],\,\dim\,[S]<r$ (resp. $\dim\,[R]_J,\,\dim\,[S]_J<r$), the induction works.
%For C-3-1, we do the same argument in $\Ga_J^*(f)$.
For the other cases, we prepare a simple lemma.
\begin{Lemma}\label{existence-maximal} Assume that  $\dim\,\Gamma(f)=n-1$. Then for any face $\Delta\subset \Gamma(f)$, there exists an $(n-1)$ dimensional face $\Xi$ such that 
 $\Xi\supset \Delta$. Equivalently for any weight $P$, the closure $\overline{[P]}$ contains a strictly positive vertex $Q$.  %(Then  $\Delta(Q)\supset \Delta(P)$.)
 Respectively $\overline{[P]_J}$  contains a strictly positive vertex $Q\in \mathcal V_J^+$. 
 \end{Lemma}
The assertion is immediate from the assumption that $\dim\,\Ga(f)=n-1$ as $\Ga_+(f)$ is a $n$-dimensional convex polyhedral region and $\Ga(f)$ is the union of 
compact boundary faces.
The assertion  for $\Gamma_J^*(f)$, as we can  use $F=ff_1\cdot f_n$ instead of $f$.

%\noindent
\subsubsection{ Case C-2 and C-3-2} We start again from the inequality
$\ell_0(P)\le \eta(P)$.
Assume that $\overline{[P]}$ (respectively $\overline{[P]_J}$) is a regular boundary region.
We prove that $\eta(P)\le \eta_{max}(f)$ (respectively $\eta(P)\le \eta_{J,max}(f)$) by the induction of  $\dim\,[P]$ (resp. $\dim\,[P]_J$).
The argument is completely same in the case $[P]_J$. 
Take a strictly positive vertex $R$ in $\overline{[P]}\cap\mathcal V^+$ (resp. in $R\in \overline{[P]_J}\cap\mathcal V_J^+$), using Lemma \ref{existence-maximal}. 
%Note that
%$R\in \mathcal V^+$ (resp. $R\in {\mathcal V}_J^{++}$).
Take the segment $\overline{RP}$ and extend it further to the right so that it arrives to a boundary point of the region, say $Q$.
Then $[Q]$ is either an inner region or a regular boundary region.
%be the end point of this extended segment.
\nlind
 If $Q$ is strictly positive and $[Q]$ is an inner region, we can apply the argument of Case C-1 or C-3-1 and we consider the estimation in $[Q]$
 by the inductive argument. 
 \nlind
 Similarly if $Q$ is strictly positive and $[Q]$ is a regular boundary region, we apply the inductions assumption, as $\dim\, [P]>\dim\,[Q]$.
 \nlind
So we assume that $Q$ is not strictly positive. 
 Then $Q$ is a non-vanishing  weight vector i.e., $d(Q)=0$.
The normalized 
form of weight vectors on this segment is given as 
$\hat R_s=\hat R+s Q$ with $0\le s<\infty$ and $\hat P=\hat R_{s_0}$ for some $s_0>0$.
%$\hat R_\infty$ should be understand as $Q=\lim_{s\to \infty} \hat R_s/s$.
Put  $I=I(Q)$ and $m_I(\hat R)=\min\{\hat r_i\,|\, i\in I\}$,
$I':=\{i\,|\, \hat r_i=m_I(\hat R)\}$
and $I_R:=\{j\,|\, \hat r_j=m(\hat R)\}$.
%Assume that $m(\hat R)=\hat r_{j_0}$.
\newline
--If  $I'\cap I_R\ne \emptyset$,  $m(\hat R)=m_I(\hat R)$ and $m(\hat R_s)\equiv \hat r_{i_0}$ for any $s$ and $i_0\in I'$.
Thus
\[
\ell_0(C(t))\le \eta(\hat P)=\eta(\hat R_{s_0})=\eta(\hat R)\le \eta_{max}(f)\,(\text{resp.} \,\le \eta_{J,max}(f)).
\]
--If $I'\cap I_R= \emptyset$, i.e., $m(\hat R)<m_I(\hat R)$,  take $j_0\in I_R$ such that there exists a small positive number $\eps$ and  $m(\hat R_s)=\hat r_{j_0}+s q_{j_0}$ for $s\le \eps$. As
 $m_I(\hat R)>\hat r_{j_0}$ but  $\hat r_{j_0}+s q_{j_0}$ is monotone increasing in $s$, there exists some $s_1$ such that $m_I(\hat R)=\hat r_{j_0}+s_1 q_{j_0}$ and for $s\ge s_1$, $m(\hat R_s)=m_I(\hat R)$.
Thus $\eta(\hat R_s)$ is monotone decreasing for $0\le s\le s_1$ and  constant for $s\ge s_1$. Thus in any case we get 
\[
\ell_0(C(t))\le \eta(P)=\eta(\hat R_{s_0})
\le\eta(\hat R_0)=\eta(\hat R)\le\eta_{max}(f)\, \,(\text{resp.} \,\le \eta_{J,max}(f)).
\]
\subsubsection{Case C-3-3.} 
 This case  requires  careful choice of the line segment for the estimation. For this purpose,
we prepare the following  Proposition \ref{nice-simplex} and Lemma \ref{nice-segment}.
A strictly positive weight vector $P$ is {\em  simplicially positive } (respectively {\em $J$-simplicially positive})
if there 
exist strictly positive linearly independent vertices $P_1,\dots, P_s$  of $\Ga^*(f)$  such that 
$P_i\in \overline{[P]}, i=1,\dots, s$ (resp.  vertices $P_1,\dots, P_s$  of $\Ga_J^*(f)$  such that $P_i\in \overline{[P]_J}, i=1,\dots, s$ )
and $P$ is in the interior of the simplex
$(P_1,\dots, P_s)$.
Here by a simplex $(P_1,\dots, P_s)$,  we mean the simplicial  cone
\[\Cone\,(P_1,\dots, P_s)=\left\{\sum_{i=1}^s \la_i P_i\,|\, \la_i\ge 0\right\}.
\]
Thus a line segment $\overline{PQ}$ is equal to 
the simplex $(P,Q)$.
%Suppose that $[P]_J$ be a vanishing region.
\begin{Proposition}\label{nice-simplex}
 Let $P$ be a strictly positive weight vector. 
Then there are two possibilities.
\begin{enumerate}
\item $P$ is simplicially positive (respectively  $J$-simplicially positive).
\item
There are linearly independent vertices  $P_1,\dots, P_{q-1},\,q\le \dim\,[P]$ of $\overline{[P]}$ (resp.  linearly independent vertices $P_1,\dots, P_{q-1},\,q\le \dim\,[P]_J$ of $\overline{[P]_J}$)
 and 
a weight vector $P_q$ which is not strictly positive  so that 
$P$ is  in the interior of the simplex $(P_1,\dots, P_{q})$.
\end{enumerate}
\end{Proposition}
\begin{proof}
The assertion follows easily from the fact that $\overline{[P]}$ (resp. $\overline{[P]_J}$)
is a polyhedral convex cone.  We use induction on $r=\dim\,[P]$ (resp. $r=\dim\, [P]_J$).
As the proof is completely parallel,   we show the assertion  in the case of $\Ga^*(f)$.
Take a strictly positive vertex $P_1\in \overline{[P]}$ using Lemma  \ref{existence-maximal}
and take the line segment $\overline{P_1P}$ and extending to the right, put $Q_1$ be the weight on the boundary of $ \overline{[P]}$. 
Thus $P$ is contained in the interior of $\overline{P_1Q_1}$.
Consider $[Q_1]$. Then $\dim\, [Q_1]<r$.
If $Q_1$ is not strictly positive, we stop the operation. Then $q=2$ and this case  corresponds  to Case (2).
 If $Q_1$ is still strictly positive but not a vertex, 
we repeat the argument on $\overline{[Q_1]}$.
Take a strictly positive vertex $P_2\in\overline{[Q_1]}$ and so on.
Apply an inductive argument.
The operation stops if we arrive at  a weight vector which is not strictly positive (then case (2))
or a strictly positive vertex $P_q$  (Case (1)).
\end{proof}
Using this proposition,  we have the following choice of a nice line segment.
\begin{Lemma} \label{nice-segment}
  Assume that $P$ is a strictly positive weight.  If $P$ is not  simplicially positive (respectively
  not $J$-simplicially positive),
 there is a line segment $\overline{RQ}$ such that
$R$  is  simplicially positive (resp. $J$-simplicially positive) and $Q$ is not strictly positive.
\end{Lemma}
\begin{proof} We give the proof for $\Ga^*(f)$ as the proof is completely parallel for $\Ga_J^*(f)$. Assume that $P$ is not simplicially positive.
Using Proposition \ref{nice-simplex}, we suppose that $P$ is in the interior of the simplex $(P_1,\dots, P_q)$
where $P_1,\dots, P_{q-1}$ are strictly positive vertices and 
$P_q$ is not strictly positive. Write $P$ by a barycentric coordinates as
$P=\sum_{i=1}^q \la_i P_i$ with $\la_i>0$ and  we may assume $ \sum_{i=1}^q \la_i=1$.
Put $\la:=\sum_{i=1}^{q-1} \la_i$,  $\mu:=1-\la$ and 
define  $R:=(\sum_{i=1}^{q-1} \la_i P_i)/\la$ and $Q:=P_q/\mu$.
Then $P=\la R+\mu Q$ and  $R\in (P_1,\dots, P_{q-1})$.  Thus $R$ is simplicially positive.
%The proof of the assertion II is exactly the same.
\end{proof}
\begin{Remark} For the proof of Case 3-3 below, the Jacobian dual Newton diagram is
essential. So we use Lemma \ref{nice-segment} for $\Ga_J^*(f)$.
\end{Remark}
{\bf Proof of Case 3-3}.
Now we are ready to have an estimation for the \L ojasiewicz exponent of our test curve $C(t)$. Suppose that $[P]$ and $[P]_J$ are  non-banishing boundary region. We apply Lemma \ref{nice-segment}.
If $P$ is  $J$-simplicially positive, we have the estimation $\ell_0(C(t))\le \eta_{J,max}(f)$
by the same argument as in Case 3-1. Thus using Lemma \ref{nice-segment},  we may assume that  $P$ is in the line segment $\overline{RQ}$ where $R$ is  $J$-simplicially positive and $Q$ is not strictly positive.
 If $Q$ is a non-vanishing weight, we proceed as the case C-3-2
to get the estimation $\ell_0(C(t))\le \eta_{J,max}(f)$.
Thus we assume that $Q$ is a vanishing weight vector and 
$d(Q,f)>0$. Assume that $Q=(q_1,\dots, q_n)$ and put  $I:=\{i\,|\, q_i=0\}$.
We assume $I=\{1,\dots, m\}$ for simplicity.
Note that $\mathbb C^I$ is a vanishing coordinate subspace.
For each $i\in I$, there exists some $j\not\in I$ and a monomial $z_i^{n_{i,j}}z_j$ with a non-zero coefficient, as $f$ has an isolated singularity at the origin.
Put $J_i$ be the set of such $j$ for a fixed $i\in I$ and put $J(I)=\cup_{i\in I}J_i$.
Here $n_{i,j}$ is assumed to be the smallest when $j$ is fixed.
%There might also some monomials of the type $c_{\nu_I,j}\bfz_I^{\nu_I}z_j$ with non-zero coefficient $c_{\nu_I,j}$. 
Put $\xi_{I}:=\max\{n_{i,j}\,|\, i\in I,j\in J_i\}$ and $\xi(f)$ be the maximum of $\xi_{I}$ where $I$  corresponds to  a vanishing coordinate subspace.
Put $\eta'_{J,max}(f):=\max\{\eta'_{k,i}(R)\,|\,R\in  \mathcal V^{++},\,1\le k,i\le n\}$
where $\eta_{j,i}'(R)={d(R,f_i)}/{r_j}$. Under the above situation, 
we will prove that  
\begin{eqnarray*}
(\star)\quad \ell_0(C(t))\le\max\{\xi(f),\eta_{J,max}(f),\eta'_{J,max}\}.\
\end{eqnarray*}
%by  the induction of $\dim\,[P]$.
%Put$\tilde J$ be the union of $J$ and such $j$.
Consider the normalized weight vector
$\hat R_s:=(1-s)\hat R+s\hat Q,\,0\le s\le 1$. 
Note that $\hat R_0=\hat R,\,\hat R_1=\hat Q$ and putting 
$\hat R_s=(\hat r_{s,1},\dots,\hat r_{s,n})$,
\[\begin{split}
%&\hat R_s=(\hat r_{s,1},\dots,\hat r_{s,n}),\, \hat R_0=\hat R,\,\hat R_1=\hat Q\\
&\hat r_{s,i}=
\begin{cases} &(1-s)\hat r_i,\quad 1\le i\le m\\
&(1-s)\hat r_{i}+s\hat q_{i},\quad m<i\le n.
\end{cases}\\
\end{split}
\]
 Put  $I'=\{i\in I\,|\, \hat r_i=m(\hat R_I)\}$
 and $J'=\cup_{i\in I'}J_i$.  
 Thus for $i\le m$, the normalized weight $\hat r_{sj}$ goes to $0$, when  $s$  approaches to $ 1$. On the other hand,  for  $j>m$,  $\hat r_{sj}\ge \delta, 0\le \forall s\le 1$ for some $\delta>0$.
 Thus there exists an $\eps,\,1>\eps>0$ so that for
 $1-\eps\le s\le 1$, $m(\hat R_s)$ is taken by $i\in I'$.
 Note that for $j\in  J'$,  there exists a  small enough $\eps_2,\,\eps_2\le \eps_1$ so that  $(f_j)_{\hat R_s}=((f_j)^I)_{\hat R_{s}}$ as 
 $\hat r_{sj}\ge \de $ for $j>m$ for 
  $1-\eps_2\le s\le 1$.  Here $(f_j)^I$ is the restriction of $f_j$ to $\mathbb C^I$ and 
 $(\hat R_{s})_I$ is the $I$ projection of $\hat R_{s}$ to $N_+^I$.
 %$(\hat R_{s})_I\in {N_+^I}$.
  That is, $(f_j)_{\hat R_s}$  contains only variable $z_1,\dots, z_m$.
By the \L ojasiewicz non-degeneracy, there exists $i_0\in I'$ and $j_0\in J_{i_0}$
 such that 
 \[
 (f_{j_0})_{\hat R_s}(\mathbf a)\ne 0,\quad s\ge1-\eps_2.
 \]
 %  by Lemma\,\ref{Lojasiewicz-non-degeneracy}.
 By the definition of Jacobian dual Newton diagram,  for  any $0<s<1$,
$(f_{j_0})_{\hat R_s}$ does not depend on $s$
and thus $(f_{j_0})_{\hat R_s}(\bfa)\ne 0$ for any $0<s<1$. Then 
$d(\hat R_s,f_{j_0})\le  (1-s)\hat r_{i_0}\times n_{i_0,j_0}$ for any $0<s\le 1$.
The equality takes place if the monomial $z_{i_0}^{n_{i_0,j_0}}z_{j_0}$ is on
the face function $(f_{j_0})_{\hat R_s}(\bfz)$.
Assume  that 
 $\hat P=\hat R_{s_0}$, $0< s_0<1$.
Thus we start from the estimation 
\begin{eqnarray}\ell_0(C(t))\le d(\hat R_{s_0},f_{j_0})/m(\hat R_{s_0}).\end{eqnarray}

 First  for $s\ge 1-\eps_2$,
 we see that 
 \[\begin{split}
 \ell_0(\hat R_s)\le
 \frac{(1-s)\hat r_{i_0}\, n_{i_0,j_0}}{(1-s)\hat r_{i_0}}
=n_{i_0,j_0}.
\end{split}
 \]
($s_0=\mu$ in the proof of  Lemma \ref{nice-segment}.)
 Put $I_R:=\{k\,|\, \hat r_k=m(\hat R)\}$.
 \newline
 (a) If  there exists a $ k\in I'\cap I_R$ i.e. $m(\hat R)=m(\hat R_I)$, we have the estimation 
 $\ell_0(\hat R_s)\le n_{i_0,j_0}$ for any $s$. In particular, $\ell_0(C(t))\le n_{i_0,j_0}$.
 \newline
 (b) Assume that $I'\cap I_R=\emptyset$ i.e. $m(\hat R)<m(\hat R_I)$.
 Choose $k\in I_R$. Then there exists a small number $\eps>0$ so that 
 $m(\hat R_s)=\hat r_{s,k}$ for $0\le s\le \eps$.
 By the definition of $I'$, this implies that $k\not \in I$.
 There exists $0<s_1<1 $ such that $m(\hat R_{s_1})=(1-s_1)\hat r_k+s_1\hat q_k=(1-s_1)\hat r_{i_0}$.
 We divide this case into two subcases.
 \newline
 (b-1) Assume that $(1-s)\hat r_k+s\hat q_k$ is monotone increasing in $s$. Then
 \[\begin{split}
 &\ell_0(P)=\ell_0(\hat R_{s_0})\le \eta(\hat R_{0})=\eta(\hat R)\le \eta_{J,max}(f),\,\,\text{if}\,\,s_0\le s_1.\\
 &\ell_0(P)=\ell_0(\hat R_{s_0})\le \frac{(1-s_0)\hat r_{i_0}n_{i_0,j_0}}{(1-s_0)\hat r_{i_0}}=n_{i_0,j_0},\,\,\text{if}\, s_0>s_1.
 \end{split}\]
 (b-2) Assume that $(1-s)\hat r_k+s\hat q_k$ is monotone decreasing in $s$. Then
 \begin{eqnarray}
\ell_0(\hat R_{s_0})&\le  \frac{(1-s_0)\hat r_{i_0}n_{i_0,j_0}}{(1-s_0)\hat r_k+s_0\hat q_k}\le \frac{\hat r_{i_0}n_{i_0,j_0}}{\hat r_k}\le\eta'_{k,j_0}(R),
%\frac{(1-s_1)\hat r_i n_{i,j_0}}{(1-s_1)\hat r_i}=n_{i,j_0},
\quad &\text{if}\,s_0\le s_1\\
\ell_0(\hat R_{s_0})&\le  \frac{(1-s_0)\hat r_{i_0} n_{i_0,j_0}}{(1-s_0)\hat r_{i_0}}=n_{i_0,j_0},\,\quad& \text{if}\,s_0\ge s_1.
 \end{eqnarray}
 where the first estimation (19) for $s_0\le s_1$ follows from the fact that 
 \[
 s\mapsto \eta_{k,j_0}'(\hat R_{s})=\frac{(1-s)\hat r_{i_0} n_{i_0,j_0}}{(1-s)\hat r_k+s\hat q_k},\quad 0\le s\le 1
 \]
  is monotone decreasing on $s$.
  Then we apply Lemma \ref{monotone2} to get the estimation
  \[
  \ell_0(C(t))\le \eta_{k,j_0}'(R)\le\eta_{J,max}'(f).
  \]
  Thus we have proved the estimation $(\star)$.
  Assuming the next lemma for a moment,  we can ignore the first term $\xi(f)$ in $(\star)$ and  the estimation reduces to
 \[\ell_0(C(t))\le \eta''_{J,max}(f),\quad \text{if}\,\,\mathbf z(t)\in \mathbb C^{*n}.
 \]
 Furthermore if  $\mathcal V_J^{++}=\emptyset$, $\eta_{i,j}'(P)=\eta_{i,j}(P)$ for $P\in \mathcal V^+$,
  as $f_P(\bfz)$ contains all the variables and therefore $\eta_{i,j}'(P)=\eta_{i,j}(P)\le \eta(P)$ or 
  $\eta_{J,max}'(f)\le\eta_{J,max}(f)=\eta_{max}(f)$.
\begin{Lemma}\label{estimation-vanishing}
The following inequality holds.
\[n_{i,j}\le \eta_{max}(f),\,\forall i\in I,\,\forall j\in J_i.\]
\end{Lemma}
\begin{proof}
Let $\Xi$ be a maximal face of $\Ga(f)$  for which the face function $f_\Xi(\bfz)$ contains  the monomial $z_{i}^{n_{i,j}}z_{j}$. %Thus non-zero coefficients 
The corresponding weight vector  $P$ (i.e., $\Delta(P)=\Xi$) is in $\mathcal V^+\subset\Ga^*(f)$ such that 
$\Xi$ is subset of the hyperplane 
\[p_1\nu_1+\dots+p_n \nu_n=d(P)\]
 and  $f_P(\bfz)$ contains the monomial $z_{i}^{n_{i,j}}z_{j}$.
 Thus $n_{i,j}p_i+p_{j}=d(P)$.
Consider an analytic curve $\bfz(t)$ corresponding to the weight $P$. Then we have
\[
n_{i,j}=\frac{d(P)-p_{j}}{p_{i}}=\eta_{i,j}(P)\le \eta(P)\le \eta_{max}(f).
\]
\end{proof}
%Note that the right side is $\eta_{max}(f)$, not $\eta_{J,max}(f)$.

To complete the proof, we have to consider the  case where the test curve is in a proper coordinate subspace.
%%%%%%%%%
\subsubsection{ Test curves in a proper subspace} \label{proper-subspace}
We consider the situation of the test curve $\bfz(t)$ defined in (\ref{test-curve2})
for which $I^c\ne \emptyset$.  Recall that $I=\{i\,|\, z_i(t)\not\equiv 0\}$ and we assume $I=\{1,\dots, m\},\,m<n$ for simplicity.

{\bf Case 1} Assume that $ f^I\not \equiv 0$. Recall that 
$P=(p_1,\dots, p_m)$ and $\bfa=(a_1,\dots, a_m)\in \mathbb C^{*I}$.
Let $\De_1=\De(P)\subset \Ga^*(f^I)$.
Consider the weight vector $\tilde P=(p_{1},\dots,p_m,K,\dots, K)\in N_+$ where $K$ is sufficiently large so that 
$\De(\tilde P,f)=\De_1$, $d:=d(P,f^I)=d(\tilde P,f)$ and $m(\tilde P)=m(P)$.
Put $\tilde {\bfa}=(\bfa,1,\dots,1)$.
Consider  $\theta_k:={\ord}\, {\partial  f}/{\partial k}(\bfz(t))$
and put $\theta:=\min\{\theta_k\,|\, \theta_k\ne \infty\}$.
Here $\theta_k=\infty$ if $ {\partial  f}/{\partial k}(\bfz(t))\equiv 0$ by definition.
Then $\ell_0(C(t))={\theta}/{m(P)}$.
Consider the modified curve
\[
\tilde  C(t): \tilde \bfz(t)=(z_1(t),\dots, z_m(t), t^N,\dots, t^N).
 \]
 Taking $N$ sufficiently large, say $N>\max\{\theta_k\,|\, \theta_k\ne \infty\}$, it is easy to see that 
 ${\ord}\,{\partial  f}/{\partial k}(\tilde \bfz(t))=\theta_k$ for  any $k$
 with $\theta_k\ne \infty$.
 Thus for such an $N$, $\ell_0(C(t))=\ell_0(\tilde C(t))$.
 Combining with  the previous argument,  we get 
 $\ell_0(C(t))\le \eta_{J,max}''(f)$.
 
{\bf Case 2} Assume that $f^I\equiv 0$. This implies $\mathbb C^I$ is a vanishing coordinate subspace.  Thus each monomial in the expansion of $f$ must contain one of
$\{z_j\,|\, j\in I^c\}$.
Thus $f_i(\bfz(t))\equiv 0$ for $i\in I$. 
Put $m(P):=\min\,\{p_i\,|\, i\in I\}$ and $I'=\{i\in I\,|\, p_i=m(P)\}$ and $J_i$ be the set of $j\in I^c$ such that a monomial $z_i^{n_{i,j}}z_j$ exists.
Then  by the \L ojasiewicz non-degeneracy, there exists $i_0\in I'$ and $j_0\in J_{i_0}$ such that  $f_{j_0}(\bfa)\ne 0$ and thus 
\[
\frac{{{\ord}}\,\partial f({\bfz}(t))}{{{\ord}}\,\bfz(t)}\le \max\{n_{i,j}\,|\, i\in I', j\in I_{i}\}\le\eta_{J,max}(f).
\]
This completes the proof of Theorem \ref{main-theorem}.
%%%%%%%%%%%%%%%%%%%
\begin{Remark}It is possible  to have $\eta_{J,max}=\eta_{\max}$
or $\eta_{J,max}''=\eta_{J,max}$ in some cases.
 For example, see the next section.
In Example \ref{example-Jdual},  we have the equality $\eta_{J,max}=\eta_{max}$ as 
the simplex $(R,e_1,e_2)$ is a regular boundary region. In fact, we have
$\eta(R)=11,  \eta(P)=9$.
\end{Remark}
\subsection{Weighted homogeneous polynomials}
We consider a weighted homogeneous polynomial with isolated  singularity at the origin. There are nice results by 
Abderrhmane \cite{Ab1} and Brzostowski \cite{Br}. I thank to Tadeusz Krasi\'nski for  informing me these papers.
Here we will give a slightly different proof as a special case of Main theorem.
\begin{Theorem}\label{estimation-weighted}
Let $f(\bfz)$ be a non-degenerate, \L ojasiewicz non-degenerate weighted homogeneous polynomial with isolated singularity at the origin
and $\dim\,\Ga(f)=n-1$. Let $R$ be the weight vector of $f$. Then we have the estimation
$\ell_0(f)\le \eta(R)$.
\end{Theorem}

\begin{proof}
Let $\hat R=(\hat r_1,\dots,\hat r_n)$ be the normalized weight of $f$.
Put $I_R:=\{i\,|\, \hat r_i=m(\hat R)\}$.
We work in $\Ga^*(f)$. 
Consider a strictly positive integral  weight vector $P$
and consider the test curve (\ref{test-curve2}) with $\bfz(t)\in \mathbb C^{*n}$.
So ignoring the higher terms, we may assume 
$C_{P,\bfa}: \bfz(t)=(a_1t^{p_1},\dots, a_nt^{p_n})$ with 
$\bfa=(a_1,\dots, a_n)\in \mathbb C^{*n}$ as before. 
% with $[P]$ is a vanishing region.
 We may assume that $P\not\in [R]$.
There is a  line segment the $\overline{SQ}$ guaranteed by Lemma \ref{nice-segment} so that $P\in\overline{SQ}$
where $S$ is  simplicially positive and $Q$ is not strictly positive.  
As $\mathcal V^+=\{R\}$, $\Ga(f)$ has only one face, $S=R$  and $Q$ is not strictly positive.
% as   in the proof of C-3-3.
Put $I=I(Q)$ and assume that $I=\{1,\dots, m\}$ as before. If $f^I\ne 0$, we 
know that $\ell_0(C(t))\le \eta(R)$. Thus we assume that $\mathbb C^{I}$ is a vanishing coordinate subspace. For each $i\in I$, there exists a monomial $z_i^{n_{i,j}}z_j$ wit $j\not\in I$.
Put $J_i$ be the set of such $j$ for a fixed $i\in I$ and put $J=\cup_{i\in I}J_i$.
%Here $n_{i,j}$ is assumed to be the smallest when $j$ is fixed.
Consider the normalized weight vector
$\hat R_s:=(1-s)\hat R+s\hat Q,\,0\le s\le 1$. Assume that $\hat P=\hat R_{s_0},\,0< s_0<1$ as before.
Note that 
\[\begin{split}
&\hat R_s=(\hat r_{s,1},\dots,\hat r_{s,n}),\, \hat R_0=\hat R,\,\hat R_1=\hat Q\\
&\hat r_{s,i}=
\begin{cases} &(1-s)\hat r_i,\quad 1\le i\le m\\
&(1-s)\hat r_{i}+s\hat q_{i},\quad m<i\le n.
\end{cases}\\
\end{split}
\]
 Thus for $i\le m$, the normalized weight $\hat r_{sj}$ goes to $0$, as $s\to 1$, while for  $j>m$,  $\hat r_{sj}\ge \delta>0\,(0\le \forall s\le 1)$ for some $\delta>0$.
 Put $I':=\{i\in I\,|\, \hat r_i=m(\hat R_I)\}$.
 There exists a positive number $\eps$ such that 
 $(f_j)_{\hat R_s}=(f_j^I)_{(\hat R_s)_I}$ and it contains only variables $z_1,\dots, z_m$ for $1>s\ge1-\eps$.  By the \L ojasiewicz non-degeneracy, there exists $i_0\in I'$ and $j_0\in J_{i_0}$ so that 
 $(f_{j_0})_{\hat R_s}(\bfa)=0$ for $1>s\ge 1-\eps$.
 Note that for any $j$,  $\De(\hat R,f_{j})=\Ga(f_{j})$, as $f_{j}$ is also weighted homogeneous of degree $d(R)-r_j$ under the weight $R$.
 Thus  $\De(\hat R, f_{j})\supset \De(\hat Q,(f_{j}))\cap\Ga(f)$. In particular, we have
 \[\begin{split}
 \De(\hat R_s,f_{j})&=\De(\hat R,f_{j})\cap \De(\hat Q,f_{j})\cap \Ga(f)\\
 &=\De(\hat Q,f_{j})\cap \Ga(f)\subset \De(\hat R,f_j),\, \forall j.
 \end{split}\]
 Thus $R, Q$ are $J$-admissible and $\Ga^*(f)=\Ga_J^*(f)$.
 
 First asssume that  $I'\cap I_R\ne\emptyset$.
 Then we have $m(\hat R_s)=(1-s) \hat r_{i_0}$
 for any $s$ and $ \ell_0(P)\le \eta_{i_0,j_0}'(P)\le n_{i_0,j_0}$.
 
 Next we assume that $I'\cap I_R=\emptyset$.
 Then $m(\hat R_s)=\hat r_{k,s}$ for some $k\in I_R$ and small $s, 0\le s\le\exists \eps'$.
  Then there exists $s_1$ such that $m(\hat R_{s_1})=(1-s_1)\hat r_{k}+s_1\hat q_{k}=(1-s_1)\hat r_{i_0}$.
 We divide this case into two subcases as before.
 \newline
 (b-1) Assume that $(1-s)\hat r_{k}+s\hat q_{k}$ is monotone increasing in $0\le s\le 1$. Then
 \[\begin{split}
 &\ell_0(P)=\ell_0(\hat R_{s_0})\le \eta(\hat R_{0})=\eta(\hat R),\,\,\text{if}\,\,s_0\le s_1\\
 &\ell_0(P)=\ell_0(\hat R_{s_0})\le \frac{(1-s_0)\hat r_{i_0}n_{i_0,j_0}}{(1-s_0)\hat r_{i_0}}=n_{i_0,j_0},\,\,\text{if}\, s_0>s_1.
 \end{split}\]
 (b-2) Assume that $(1-s)\hat r_{k}+s\hat q_{k}$ is monotone decreasing in $s$. Then
 \begin{eqnarray*}
\ell_0(\hat R_{s_0})&\le  \frac{(1-s_0)\hat r_{i_0}n_{i_0,j_0}}{(1-s_0)\hat r_{k}+s_0\hat q_{k}}
&\le \frac{\hat r_{i_0}n_{i_0,j_0}}{\hat r_{k}}\\
&\le \eta'_{{k},i_0}(\hat R)&=\eta_{k,i_0}(\hat R)\le \eta(\hat R),
%\frac{(1-s_1)\hat r_i n_{i,j_0}}{(1-s_1)\hat r_i}=n_{i,j_0},
\quad \text{if}\,s_0\le s_1\\
\ell_0(\hat R_{s_0})&\le \frac{(1-s_0)\hat r_{i_0} n_{i_0,j_0}}{(1-s_0)\hat r_{i_0}}&=n_{i_0,j_0},\,\quad \text{if}\,s_0\ge s_1.
 \end{eqnarray*}
 where the first estimation for $s_0\le s_1$ follows from the fact that 
 \[
 s\mapsto \frac{(1-s)\hat r_i n_{i,j_0}}{(1-s)\hat r_{k}+s\hat q_{k}},\quad 0\le s\le 1
 \]
  is monotone decreasing on $s$. Thus in any case, $\ell_{0}(C(t))\le \eta(\hat R)$.
  For $P$ which is not strictly positive, the proof is the exactly  same as holomorphic case.
  %the assertion is proved.
\end{proof}
\begin{Remark}
For a non-degenerate weighted homogeneous polynomial $f(\bfz)$ with isolated singularity at the origin,
it might also be possible  that $f$ satisfies the \L ojasiewicz non-degeneracy
for all or almost all such polynomials. We leave the further discussion to the readers.
For the following typical  simplicial weighted homogeneous polynomials, \L ojasiewicz non-degeneracy is checked directly.
\begin{eqnarray*}
f_B(\bfz)&=&z_1^{a_1}+\dots+z_{n-1}^{a_{n-1}}+z_n^{a_n}\\
f_I(\bfz)&=&z_1^{a_1}z_2+\dots+z_{n-1}^{a_{n-1}}z_n+z_n^{a_n}\\
f_{II}(\bfz)&=&z_1^{a_1}z_2+\dots+z_{n-1}^{a_{n-1}}z_n+z_n^{a_n}z_1.\\
\end{eqnarray*}
These polynomials appear in the classification of weighted homogeneous polynomials
 by Orlik-Wagreich \cite{OW}.
The corresponding mixed polynomials are studied in \cite{Inaba} in the problem of isotopy construction to holomorphic links.
\end{Remark}
%%%%%%%%%%
\subsection{When the equality  $\ell_0(f)=\eta_{max}(f)$ holds?}
Suppose that 
$f$ is a  non-degenerate, \L ojasiewicz non-degenerate weighted homogenous polynomial with isolated singularity at the origin
and let $R=( r_1,\dots, r_n)$ be the integral weight vector.
Assume $r_i\le  r_{i+1}$ for $i=1,\dots, n-1$ for simplicity.
Assume also that  $ r_1=\dots= r_k<r_{k+1}$.
 If 
there exists an $i_0\le k$ such that $a_{i_0}\ne 0$ and there exists a polar curve
$\bfz=\bfz(t)$ such that 
\[
\Gamma:\begin{cases}&z_i(t)=a_it^{r_i},\,i=1\dots, n\\
&\frac{\partial f}{\partial z_j}(\bfa)=0,\,\forall j\ne i_0.
\end{cases}
\]
Then ${\ord}\, \bfz(t)=p_{1}=m(R)$ and ${\ord}\,\partial f(\bfz(t))=d-p_1$.
Thus we have $\ell_0(C(t))=\ell_0(f)=d/p_1-1$. For $n=3$,  there is a affirmative result
 by \cite{Krasinski}.
 The following example shows that in general, we can not take such a polar curve with
  $\bfa\in \mathbb C^{*n}$.
\begin{Example}
1. Consider $f(\bfz)=z_1^2z_2+z_2^3z_3+z_3^4z_1+z_4^2$.
Then $f$ is a weighted homogeneous  polynomial and $\mathcal V^+$ has a single vertex
$P=(9/25,7/25,4/25,1/2)$.  We see that $\ell_0(f)=\eta_{max}(f)=21/4$.
This number  is taken  on the  face of $f^I$, $I=\{1,2,3\}$ by the family
\[=\root 4\of {\frac 34}it^9,\,z_2(t)=\root 4\of{\frac 1{12}}i t^7,
\, z_3(t)= -t^4,\, z_4(t)\equiv 0.
\]
and ${\ord}\, \bfz(t)=4,\,{\ord}\,\partial f(\bfz(t))=21$. 
As this example shows, the maximal \L ojasiewicz number $\ell_0(C(t))$ need not to be taken on a vertex $R\in \mathcal V^+$.
Consider $g(\mathbf z')=z_1^2z_2+z_2^3z_3+z_3^4z_1$. Then $f$ is a join of two functions
$f(\mathbf z)=g(\mathbf z')+z_4^2$ and $\ell_0(f)=\ell_0(g)$ by Join Theorem below.
\end{Example}
\subsection{\L ojasiewicz Join Theorem}
Consider a join type function $f(\bfz,\bfw)=g(\bfz)+h(\bfw)$ where $\bfz\in \mathbb C^n$ and $\bfw\in \mathbb C^m$. 
Assume that 
both $g(\bfz)$ and $h(\bfw)$ have 
%non-degenerate Newton boundaries and 
isolated singularities at the respective origin.
Then
\begin{Lojasiewicz Join Theorem} \label{Lojasiewicz Join Theorem}\rm{(\cite{Te}, Corollary 2, \S 2)}
We have the equality:
\[\ell_0(f)=\max\{\ell_0(g),\ell_0(h)\}.\]
\end{Lojasiewicz Join Theorem}
\begin{proof}
%We may assume $\ell_0(g)\le \ell_0(h)$. 
Put $\bfu=(\bfz,\bfw)\in \mathbb C^{n+m}$.
Take any analytic curve $C(t): \bfu(t)=(\bfz(t),\bfw(t)),\,0\le t\le 1$.

Case 1. If $\bfz(t)\equiv 0$ (respectively  $\bfw(t)\equiv 0$), 
$\frac{{\ord}\, \partial f(\bfu(t))}{{\ord}\,\bfu(t)}\le \ell_0(h)$ (resp. $\le\ell_0(g)$).

Case 2. Assume that $\bfz(t)\not\equiv 0$ and  $\bfw(t)\not\equiv 0$.
If ${\ord}\,\bfz(t)\le {\ord}\,\bfw(t)$, we have
\[\begin{split}
\frac{{\ord}\,\partial f(\bfu(t))}{{\ord}\,\bfu(t)}&=\min\left \{\frac{{\ord}\partial g(\bfz(t))}{{\ord}\,\bfz(t)},
\frac{{\ord}\partial h(\bfw(t))}{{\ord}\,\bfz(t)}\right\}\\
&\le \min\left\{\ell_0(g),\frac{\ell_0(h){\ord}\,\bfw(t)}{{\ord}\,\bfz(t)}\right\}\\
&\le \ell_0(g).
\end{split}
\]
If ${\ord}\,\bfz(t)>{\ord}\,\bfw(t)$, by  the same argument,
\[\begin{split}
\frac{{\ord}\,\partial f(\bfu(t))}{{\ord}\,\bfu(t)}
&\le \ell_0(h).
\end{split}
\]
Thus  we have $\ell_0(f)\le \max\{\ell_0(g),\ell_0(h)\}$.
The equality can be taken by a curve $\bfu(t)=(\bfz(t),{\bf 0})$ or $\bfu(t)=(0,\bfw(t))$
which takes $\ell_0(g)$ for $g(\bfz)$
or $\ell_0(h)$ for $h(\bfw)$ respectively.
\end{proof}
\begin{Remark}
As we see in the proof of Theorem \ref{main-theorem}, it is not necessary to take the Jacobian dual Newton diagram everywhere.
We only need consider $\Gamma_{J}^*(f)$ in the vanishing regions of $\Gamma^*(f)$. Namely if $P$ is a vertex of $\Ga_J^*(f)$ which is in an  inner or a regular boundary region of $\Ga^*(f)$, we have an estimation
$\eta(P)\le \eta_{max}(f)$.
\end{Remark}
%%%%%%%%%%%%%%%%%%%%%%%%%%%%%
\subsection{Application}
We consider the hypersurface $V=f\inv(0)$ and the transversality problem with the sphere
$S_r:=\{\bfz\,|\, \|\bfz\|=r\}$. Certainly transversality  has been shown by Milnor\cite{Milnor}. However   we want to see this property from a slightly different view
point.
Recall first that 
$S_r$ and $V$ does not  intersect transversely at $\bfz=\bfa$ if and only if
\newline
(i) %(holomorphic case)
$\bfa$ and $\overline{ \partial f(\bfa)}$ are linear dependent over $\mathbb C$, or
\nl
Using hermitian inner product and the Schwartz inequality, this condition is equivalent to
\newline\noindent
(ii) $ |(\bfa,\overline{ \partial f(\bfa)})_{norm}|=1.$

\subsubsection{Orthogonality at the limits (Whitney (b)-regularity)}

\begin{Lemma}\label{tangent-cone}
 Assume that $f$ is a non-degenerate and \L ojasiewicz non-degenerate  holomorphic  function with an isolated singularity at  the origin.
Consider a non-constant analytic curve $\bfz(t)$ with $\bfz(0)={\bf 0}$
defined as (\ref{test-curve}) and asssume that $f_P(\bfa)=0$.  Then we have
%\nl
\begin{enumerate}
\item
$\lim_{t\to 0} (\bfz(t),\overline{\partial f(\bfz(t)})_{norm}=0$.
Geometrically this implies the limit direction of $\bfz(t)$ is contained in
the limit of  the tangent space $T_{{\bfz}(t) }V$ of  $V$. 
\item 
In particular, there exists a positive number $r_0$ so that the sphere $S_r$ intersects $V$ transversely for any $r\le r_0$.
\end{enumerate}
\end{Lemma}
\begin{proof}
Let $J=\{j\,|\, z_j(t)\equiv 0\}$ and let $I$ be the complement of $J$.
We assume that $I=\{1,\dots,m\}$.  Consider the Taylor expansion
\[
z_i(t)=a_it^{p_i}+\text{(higher terms)},\quad i\in I.
\]
Put $P=(p_i)_{i\in I}\in N_+^I$ as before.

Case 1. Assume that $f^I\not\equiv 0$ and assume for simplicity
$p_1=\dots=p_{\ell}<p_{\ell+1}\le\dots\le p_m$ and put $d=d(P,f^I)$.
Note that ${\ord}\,\bfz(t)=p_1$. Put $q={\ord}\, \partial f(\bfz(t))$.
By the assumption, $\lim_{t\to 0} t^{-p_1}\bfz(t)=\bfa_\infty$ where 
$\bfa_\infty:=(a_1,\dots, a_\ell,0,\dots,0)$.
Put $\bfv_\infty:=\lim_{t\to 0}t^{-q}\partial f(\bfz(t))$.
By the non-degeneracy, we have  $ q\le d-p_1$.
If $q<d-p_1$, the assertion (1)
is immediate, as $\bfv_\infty\in \mathbb C^K$ where $K=\{j\,|\,j>\ell\}$.
Assume that $q=d-p_1$. 
This implies
\[\frac{\partial f^I}{\partial z_j}(\bfa)=0,\quad j\ge \ell+1.
\]
%The equality $f(\bfz(t))\equiv 0$ implies that 
By the assumption $f_P^I(\bfa)=0$ and 
 the Euler equality of $f_P^I(\bfz)$, 
we get 
\[
f_P^I(\bfa)=\sum_{i=1}^m p_i a_i\frac{\partial f_P^I}{\partial z_i}(\bfa)=
p_1\sum_{i=1}^\ell  a_i\frac{\partial f_P^I}{\partial z_i}(\bfa)=0.
\]
In this case, note that $\bfv_\infty=(\frac{\partial f_P^I}{\partial z_1}(\bfa),\dots, \frac{\partial f_P^I}{\partial z_\ell}(\bfa),0,\dots,0)$.
This implies also 
\[\begin{split}
&\lim_{t\to 0}(\bfz(t),\overline{\partial f}(\bfz(t)))_{norm}=({\bfa}_{\infty},
\bfv_\infty)_{norm}\\
&=c\sum_{i=1}^\ell a_i\frac{\partial f_P^I}{\partial z_i}(\bfa)=0
\end{split}
\]
where $c$ is a non-zero scalar.
%and the assertion is clear.

Case 2.  Assume that $f^I\equiv 0$. Then $\bfz(t)\in \mathbb C^I$
and $\mathbb C^I$ is a vanishing coordinate subspace, and thus 
$\partial f(\bfz(t))\in \mathbb C^{I^c}$. Thus the assertion is obvious. The assertion (2) follows from (1).
Of course, (2) is nothing but the existence of  a stable radius which is well known by \cite{Milnor} for a general holomorphic function with an islated singularity at the origin.
\end{proof}
\begin{Remark}
The assertion of the lemma says that  the stratification $\{V\setminus\{\mathbf 0\}, \{\mathbf 0\}\}$ is a Whitney $b$-regular stratification.
\end{Remark}
%%%%%
\subsubsection{Making $f$ convenient without changing the topology}
Let $f$ be a non-degenerate, \L ojasiewicz non-degenerate function with isolated singularity at the origin. Choose  integer $N_i$ with
\begin{eqnarray}\label{make-convenience}
 N_i> \eta_{J,max}''(f)+1,\quad i=1,\dots,n
 \end{eqnarray}
 and consider  a polynomial 
\[R(\bfz)=c_1z_1^{N_1}+\cdots+ c_nz_n^{N_n}.
\]
Consider the family of functions:
\[
f_s(\bfz)=f(\bfz)+s R(\bfz),\quad 0\le s\le 1.
\]
The coefficients are chosen generically so that $f_s$ is non-degenerate.
Note that $f_0=f$ and $f_1$ is a convenient and non-degenerate 
\begin{Theorem}\label{making-convenience}
Consider the family of  hypersurface $V_s:=f_s\inv(0),\,0\le s\le 1$.
There exists a positive number $r_0$ such that
$V_s\cap B_{r_0}$ has a unique singular point at the origin  and  for any $r\le r_0$
 the sphere $S_r$ and $V_s$ intersect transversely for any 
$0\le s\le 1$. In particular, the links of $f$ and $f_1$ are isotopic and their Milnor fibrations are isomorphic.
\end{Theorem}
\begin{proof} 
Take $r_0$ so that there exists a positive number $c$ and the following inequality is satisfied.
\[
\|\partial f(\bfz)\|\ge c \|\bfz\|^{\ell_0(f)},\,\,0<\| \bfz\|\le r_0.
\]
Assume that the assertion does not hold. Then we can find an analytic curve
$(\bfz(t),s(t)),\, 0\le t\le 1$ and Laurent series $\la(t)$ such that
$\bfz(0)={\bf 0}$ and $s(0)=s_0\in [0,1]$ and 
\begin{eqnarray}\label{transverse}
\overline{\partial f_{s(t)}(\bfz(t))}=\la(t)\bfz(t),\quad 
f_{s(t)}(\bfz(t))\equiv 0.
\end{eqnarray}
Expand $\bfz(t) $ and $s(t)$ in Taylor expansions
\[\begin{split}
z_i(t)&=a_it^{p_i}+\text{(higher terms)},\, i=1,\dots, n\\
s(t)&=s_0+ \beta t^b+\text{(higher terms)},\,0\le s_0\le 1
\end{split}
\]
 and let  $I=\{i\,|\, z_i(t)\not\equiv 0\}$, $P=(p_i)\in N_+^{*I}$, and $\bfa=(a_i)\in \mathbb C^{*I}$. Then by the definition of $R(\mathbf z)$,   %(\ref{make-convenience})
\begin{eqnarray}
{\ord}\,s(t)\partial R(\bfz(t))\ge m(P)(N-1) > m(P)\eta_{J,max}''%=t_0 N a_i^{N-1}t^{p_i(N-1)+b}+\text{(higher terms)}.
\end{eqnarray}
where $N=\min\{N_1,\dots, N_n\}$.
%Case 1. Assume that $f^I\not \equiv 0$. 
By Theorem\,\ref{main-theorem2},
\begin{eqnarray}\label{limit-tangent}
{\ord}\,{\partial f}(\bfz(t))\le m(P)\eta_{J,max}''.
%&=&\frac{\partial f_P}{\partial z_i}(\bfa)t^{d(P,f)-p_i}+\text{(higher terms)}
\end{eqnarray} 
Thus 
\[
{\ord}\,{\partial} f_{s_0}(\bfz(t))={\ord}\,\partial f(\bfz(t))
\]
and 
\[
\lim_{t\to 0} (\partial f_{s(t)}(\bfz(t))_{norm}=\lim_{t\to 0} (\partial f_{s(0)}(\bfz(t))_{norm}.
\]
This already implies the family of hypersurfaces $V_s$ have isolated singularities at the origin.
Note that
the equality 
$f_P(\bfa)=0$ follows from the equality $f_{s(t)}(\bfz(t))\equiv 0$.
The assumption gives us the contradicting equalities
\[\begin{split}
&|(\bfz(t),\overline{\partial f_{s(t)}(\bfz(t))})_{norm}|\equiv 1,\quad
\text{and}\,\,\\
&\lim_{t\to 0}\,{(\bfz(t),\overline{\partial f_{s(t)}(\bfz(t))})_{norm}}= 0
\end{split}
\]
where the first equality follows from the assumption  (\ref{transverse}) and the second convergence follows from (\ref{limit-tangent}) and Lemma \ref{tangent-cone}.
Thus the family $f_s,\,0\le s\le 1$ has a uniform stable radius and the isomorphisms of the Milnor fibrations  are easily obtained using tubular Milnor fibration and Ehresman's fibration theorem (\cite{W}).
\end{proof}
\begin{Remark}
By the same argument, it is easy to see that $f_t(\bfz)=f(\bfz)+t \bfz^{\nu}$ with
$\sum_{i=1}^n \nu_i\ge \eta_{J,max}''(f)+1$ does not change the topology at  the origin.
A similar result for $C^0$-sufficiency is proved in Kuo \cite{Kuo}.
\end{Remark}
%%%%%%%%%%
%\input Lojasiewicz-3-All2.tex

\newpage
\section{
%Part II. 
\L ojasiewicz inequality for mixed functions}

%\noindent
In this   section, 
we will introduce the notion of  \L ojasiewicz exponent  for a mixed function and we generalize the estimation obtained for non-degenerate holomorphic  functions in  previous sections for $f(\mathbf z,\bar{\mathbf z})$ which are strongly non-degenerate.  %and \L ojasiewicz non-degenerate mixed functions.

%\section{\L ojasiewicz inequality for strongly non-degenerate  mixed functions }
\subsection{Newton boundaries and various gradients}
Let  $f$  be a mixed function  expanded as
 \[f(\bfz,\bar\bfz)=\sum_{\nu,\mu}c_{\nu\mu}\bfz^\nu{\bar\bfz}^\mu.
 \] 
The Newton diagram $\Ga_+(f)$ is defined as the convex hull of 
\[\bigcup_{c_{\nu\mu}\ne 0}\left ((\nu+\mu)+(\mathbb R_+)^n\right)
\]
and the Newton boundary $\Ga(f)$ is defined  by the union of compact faces of $\Ga_+(f)$ as in the holomorphic case,
using the radial weighted  degree of the monomial $\bfz^\nu{\bar\bfz}^\mu$. Here the radial weight degree with respect to the weight vector $P=(p_1,\dots, p_n)$,  is defined by
\[\rdeg_P \bfz^\nu{\bar\bfz}^\mu:=\sum_{i=1}^n p_i(\nu_i+\mu_i).\]
The notion of {\em strong non-degeneracy} is introduced in  \cite{OkaMix} to study
 Milnor fibration  defined by $f$. Let us recall the definition.
Take an arbitrary face $\Delta$ of $\Gamma(f)$ of any dimension.
The face function is defined  by
$f_{\Delta}(\bfz,\bar\bfz)=\sum_{\nu+\mu\in \Delta}c_{\nu\mu}\bfz^\nu{\bar\bfz}^\mu$.
Let $P$ be the weight vector. Then $f_P$ is defined as $f_{\Delta(P)}$ as in the holomorphic case. $f_P(\bfz,\bar\bfz)$ is a radially weighted homogeneous polynomial with weight $P$.
A mixed function $f$ is called {\em strongly non-degenerate} if
$f_{\Delta}:\,\mathbb C^{*n}\to \mathbb C$ has no critical point
 for any face function $f_{\Delta}$. It is known that such a mixed function admits a Milnor fibration at the origin (\cite{OkaMix}).

We assume that $\Ga(f)$ has dimension $n-1$.
%Newton non-degenerate mixed polynomial. 
We consider two gradient vectors:
\begin{eqnarray}
&\partial f&:=(\frac{\partial f}{\partial z_1},\dots, \frac{\partial f}{\partial z_n})\\
&\bar\partial f&:=(\frac{\partial f}{\partial \bar z_1},\dots, \frac{\partial f}{\partial \bar z_n})
\end{eqnarray}

To study the behavior of the tangent spaces, it is more useful to use the real and imaginary part of $f$. Put $f=g+ih$ where $g$ and $ h$ are real and imaginary parts of $f$. Putting $z_j=x_j+iy_j$, $g, h$ are real analytic functions of $2n$ variables $x_1,y_1,\dots, x_n,y_n$. Substituting $x_j=(z_j+{\bar z_j})/2$ and $y_j=(z_j-{\bar z_j})/2i$, we consider $g,h$ as mixed functions.
As  they are real valued mixed functions,
we have 
\[
T_{\bfa}g\inv(0)=\{\bfv\,|\, \Re(\bfv,\bar\partial g)=0\},\,\quad
T_{\bfa}h\inv(0)=\{\bfv\,|\, \Re(\bfv,\bar\partial h)=0\}.
\]
Here $(\bfv,\bfw)$ is the hermitian inner product.
As we have $\overline{\partial g}=\bar\partial g$ and $\overline{\partial h}=\bar \partial h$ ( \cite{OkaAf}), various gradients are related by 
\begin{eqnarray}
\bar \partial f&=\bar\partial g+i\bar \partial h,
\,\quad& \overline{\partial f}=\bar\partial g-i\bar \partial h\\
\bar\partial g&=(\bar\partial f+\overline{\partial f})/2,
\,\,
&\bar \partial h=(\bar f-\overline{\partial f})/2i.
\end{eqnarray}
For a weight vector $P=(p_1,\dots, p_n)$, the real  part and imaginary part $g_P=\Re f_P, h_P=\Im f_P$ of $f_P$ are  real-valued radially weighted homogeneous polynomials with weight $P$ and the Euler equality take the form(\cite{OkaPolar}):
\begin{eqnarray}
g_P(\bfz,{\bar \bfz})=2\sum_{i=1}^n p_iz_i\frac{\partial g_P}{\partial z_i},\,
h_P(\bfz,{\bar \bfz})=2\sum_{i=1}^n p_iz_i\frac{\partial h_P}{\partial z_i}.
\end{eqnarray}
The strong non-degenracy implies that 
%the rank of 
$\{
\bar \partial g_P(\bfa,{\bar\bfa}),
\bar\partial h_P(\bfa,{\bar\bfa})\}$
are linearly independent over $\mathbb R$
for any $\bfa=(a_1,\dots, a_n)\in \mathbb C^{*n}$.
We consider the \L ojasiewicz inequalities of real-valued mixed functions $g$ and $h$.
 For a non-zero vector $\bfw\in \mathbb C^n$, we denote the real hyperplane orthgonal to $\bfw$ by
 $\bfw^{\perp}:=\{\bfv\in \mathbb C^n\,|\, \Re(\bfv,\bfw)=0\}$.  We denote the normalized vector
 $\bfw/\|\bfw\|$  by $\bfw_{norm}$. 
The problem (which do not happen for holomorphic functions)
is that along a given analytic curve $\bfz(t),\,0\le t\le 1$ with $\bfz(0)={\bf 0}$,
the limit directions  
$
\lim_{t\to 0}(\bar \partial g(\bfz(t)))_{norm}$  and  $\lim_{t\to 0}(\bar \partial h(\bfz(t)))_{norm}$ can be linearly dependent over $\mathbb R$.  
In this case, we have a proper inclusion:
\begin{multline}
\lim_{t\to 0}T_{\bfz(t)} V=\lim_{t\to 0}\left( (\bar \partial g(\bfz(t))^\perp\cap  (\bar \partial h(\bfz(t)))^\perp\right)\\
\subsetneq  \lim_{t\to 0}(\bar \partial g(\bfz(t))_{norm})^{\perp}\cap \lim_{t\to 0}(\bar \partial h(\bfz(t))_{norm})^{\perp}
\end{multline}
The following lemma plays a key role to solve this problem.
Let $I=\{1\le j\le n\,|\, z_j(t)\not\equiv 0\}$ and 
and $I^c$ be the complement of $I$.
\begin{eqnarray}\label{test-curve3}
C(t):  \begin{cases}
&\bfz(t)=(z_1(t),\dots, z_n(t)),\,\,\bfz(0)=0,\, \bfz(t)\in \mathbb C^{*I},\,t>0\\
&z_i(t)=a_it^{p_i}+\text{(higher terms)},\quad a_i\ne 0,\, i\in I\end{cases}
\end{eqnarray}
Consider the weight vector $P=(p_i)\in { N}^{*I}$, $\bfa=(a_i)_{i\in I}\in \mathbb C^{*I}$  and we assume that $f^I\ne 0$.
For any real valued analytic function $b(t)$ defined on an open neighborhood of $t=0$, we consider the modified gradient vectors, defined as follows.
\[
\begin{split}
&(\bar \partial g(\bfz(t)))_{b(t)}:=\bar \partial g(\bfz(t))+b(t) \bar \partial h(\bfz(t))\\
&(\bar \partial h(\bfz(t)))_{b(t)}:=\bar \partial h(\bfz(t))+b(t) \bar \partial g(\bfz(t)).
\end{split}
\]
Note that any of the following three pairs generate the same real dimension 2 subspace over $\mathbb R$ at $T_{\mathbf z(t)}\mathbb C^n$.
\[\{\bar \partial g(\bfz(t)),\bar \partial h(\bfz(t))\},
\,\{\bar \partial g(\bfz(t)),(\bar \partial h(\bfz(t)))_{b(t)}\},\,
\{(\bar \partial g(\bfz(t)))_{b(t)},\bar \partial h(\bfz(t))\}
\]
The following is a key lemma to generalize the assertions obtained in previous sections
 for mixed functions.
Consider the family of hypersurface  $V_t:=\{\bfz\in \mathbb C^n\,|\, f(\bfz,{\bar\bfz})=f(\bfz(t),{\bar\bfz(t)})\}$ for $ -\eps\le t\le \eps$
where  $V_t$ passes  through $\bfz(t)$.
\begin{Lemma} 
\label{basic-lemma}
Assume that $\ord\, \bar \partial g(\bfz(t))\le \ord\, \bar \partial h(\bfz(t))$ for simplicity.
 %Under the above situation,  we have the following.
 
\noindent
{\rm (i)} There exists a real valued  analytic function
$b(t)$ so that 
two  analytic curves 
$\bar \partial g(\bfz(t)), (\bar\partial h(\bfz(t)))_{b(t)}$ have normalized limits
\[\begin{split}
&v_{g,\infty}:=\lim_{t\to 0}(\bar \partial g(\bfz(t)))_{norm}\,\,\text{and}\\
&v_{h,\infty}':=\lim_{t\to 0} ((\bar\partial h(\bfz(t)))_{b(t)})_{norm}
\end{split}
\]
 which are  linearly indepent over 
$\mathbb R$. The limit of the tangent space $T_{\bfz(t)}V_t$ is equal to  the intersection of the hyperplanes
$v_{g,\infty}^{\perp}\cap (v_{h,\infty}')^{\perp}$.
\newline
{\rm (ii)} The orders of the vectors $ \bar \partial g(\bfz(t)),\,  (\bar\partial h(\bfz(t)))_{b(t)}$
 satisfy the inequality:
\[{\ord}\,  \bar\partial g(\bfz(t)),\,{\ord}\,  (\bar\partial h(\bfz(t)))_{b(t)}\le d(P,f^I)-m(P).
\]
{\rm (iii)} If further $f_P(\bfa)= 0$, the limit vector
  $\lim_{t\to 0}\bfz(t)_{norm}$ is real orthgonal to  %the limit of the tangent plane $T_{\bfz(t)}V$ which is equal to
$v_{g,\infty}$ and $v_{h,\infty}'$ i.e., $\lim_{t\to 0}\bfz(t)_{norm}\in v_{g,\infty}^{\perp}\cap (v_{h,\infty}')^{\perp}$.
\nl
{\rm (iv)}  For any analytic  functions $b(t), c(t)$,
\[{\ord}\,\bar\partial g(\bfz(t))_{c(t)},\, {\ord}\,\bar\partial h(\bfz(t))_{b(t)}\le d(P,f^I)-m(P).
\]
%\end{enumerate}\ord\,\bar\partial g(\bfz(t))_c(t)
\end{Lemma}
Here  two vectors $\bfv,\bfw\in \mathbb C^n$  are called to be  real orthgonal if $\Re(\bfv,\bfw)=0$.
\begin{proof} The proof is essentially the same as the proof of Theorem 3.14,  \cite{ EO13}.
For the convenience,  we repeat the proof briefly. For further related discussion, see \cite{OkaAf,EO13}. Put  $I=\{i\,|\, z_i(t)\not \equiv 0\}$ and $J=I^c$.
We may assume that $I=\{1,\dots,m\}$, $\bfa=(a_1,\dots, a_m)$ and 
\[
p_1=p_2=\cdots=p_k<p_{k+1}\le \cdots\le p_m.
\]
Put $d=d(P,f^I)$.
%We use the notation $\bfv_{norm}:=\bfv/\|\bfv\|$.
Under the above assumption, $\ord\,\bfz(t)=p_1$ and 
\[\begin{split}
\lim_{t\to 0} (\bfz(t))_{norm}=\bfa(\infty)_{norm}
\qquad \text{where} \quad\bfa(\infty):=(a_1,\dots, a_k,0,\dots,0).
\end{split}
\]
By the definition of $P$ and the assumption $f^I\ne 0$, 
\[\begin{split}
\frac{\partial g^I}{\partial \bar z_i}(\bfz(t))&=\frac{\partial g^I_P}{\partial\bar z_i}(\bfa) t^{d-p_i}+\text{(higher terms)}\\
\frac{\partial h^I}{\partial \bar z_i}(\bfz(t))&=\frac{\partial h^I_P}{\partial \bar z_i}(\bfa) t^{d-p_i}+\text{(higher terms)}
\end{split}
\]
Thus by the strong non-degeneracy assumption, we have  the inequality:
\[
%d-p_m\le
 {\ord}\,{\bar\partial g^I}(\bfz(t))\le d-p_1,\quad
%d-p_m\le
{ \ord}\,{\bar\partial h^I}(\bfz(t))\le d-p_1.
\]
For an analytic curve $\bfv(t)$ with $\bfv(0)=0$, we associate scalar vector 
\[
\be(\bfv(t)):=(\beta_1,\dots, \beta_m),
\,\,\text{where}\,\,\beta_i=\text{Coeff}(v_i(t),t^{d-p_i})
\]
and integers 
\[\begin{split}
&d(\bfv(t)):=\min\{\ord\, v_i(t)\,\, |\, i=1,\dots, m\},\,\\
&\ga_{\bfv}:=\max\{i\,|\, \ord\,v_i(t)=d(\bfv(t))\}. 
\end{split}
\]
Note that $\ga_{\bfv}$ is the largest index for which 
$\lim_{t\to 0}\bfv(t)_{norm}$ has non-zero coefficient.
We call $\ga_{\bfv}$ {\em the leading index} of $\bfv(t)$.

We start from two analytic curves  $\bar\partial g^I(\bfz(t))$ and $\bar\partial h^I(\bfz(t))$.
Put  $d_g={\ord}(\bar\partial g^I(\bfz(t))$, $\,d_h={\ord}(\bar\partial h^I(\bfz(t))$ and $\ga_{g}:=\ga_{\bar\partial g^I(\bfz(t))},\,\ga_h:=\ga_{\bar\partial h^I(\bfz(t))}$.
We assume that $d_g\le d_h$.
First we associate $2\times m$-matrix with complex coefficients by
\[
A(\bar\partial g^I,\bar\partial h^I )
:=\left(\begin{matrix}
\be(\bar\partial g^I(\bfz(t))\\
\be(\bar\partial h^I(\bfz(t))
\end{matrix}\right)=\left(
\begin{matrix}
\frac{\partial g^I_P}{\partial\bar z_1}(\bfa) &\cdots&\frac{\partial g^I_P}{\partial \bar z_m}(\bfa) \\
&&\\
\frac{\partial h^I_P}{\partial\bar z_1}(\bfa) &\cdots&\frac{\partial h^I_P}{\partial\bar z_m}(\bfa)
\end{matrix}\right)
\]
By the strong  non-degeneracy assumption,  two raw complex  vectors are linearly independent over $\mathbb R$.
The normalized limit, $\lim_{t\to 0}(\bar\partial g^I(\bfz(t)))_{norm}$, has non-zero j-th coefficient if and only if
${\ord}\,\frac{\partial g^I}{\partial{\bar z}_j}(\bfz(t))=d_g$.
The following three cases are possible.
\begin{enumerate}
\item $\ga_g\ne \ga_h$.
\item $\ga_g=\ga_h$ and  $\{{\Coeff}\,(\frac{\partial g^I}{\partial{\bar z}_{\ga_g}}(\bfz(t)),t^{d_g}),
{\Coeff}\,(\frac{\partial h^I}{\partial{\bar z}_{\ga_h}}(\bfz(t)),t^{d_h})\}$ are linearly independent over $\mathbb R$.
\item $\ga_g=\ga_h$ and  $\{{\Coeff}\,(\frac{\partial g^I}{\partial{\bar z}_{\ga_g}}(\bfz(t)),t^{d_g}),
{\Coeff}\,(\frac{\partial h^I}{\partial{\bar z}_{\ga_h}}(\bfz(t)),t^{d_h})\}$
are linearly dependent over $\mathbb R$.
\end{enumerate}
In the cases of (1),(2), their normalized limits are linearly independent over $\mathbb R$ and there is no operation necessary.
In the case of (3), we put $b_1(t)=\rho_1 t^{d_h-d_g}$ % with $\rho_1\in \mathbb R$ 
and 
put 
$(\bar\partial h)'(t):=\bar\partial h^I(\bfz(t))-b_1(t)\bar\partial g^I(\bfz(t))$. 
Here $\rho_1$ is the unique real number such that 
\[\rho_1{\Coeff}\,(\frac{\partial g^I}{\bar z_{\ga_g}}(\bfz(t)),t^{d_g})-
{\Coeff}\,(\frac{\partial h^I}{\bar z_{\ga_h}}(\bfz(t)),t^{d_h})=0.
\] 
After this operation, we have
three possible cases.
\begin{enumerate}
\item[(1)'] $\ga_{(\bar\partial h)'(t)}\ne \ga_{\bar\partial g^I(\bfz(t))}$
\item [(2)'] $\ga_{(\bar \partial h)'(t)}=\ga_{\bar \partial g^I(\bfz(t))}$ but the leading coefficients of
$ (\bar \partial h)'(t) $ and $\bar\partial g^I(\bfz(t))$
are linearly independent over $\mathbb R$.
\item[(3)'] $\ga_{(\bar \partial h)'(t)}=\ga_{\bar \partial g^I(\bfz(t))}$  and the leading coefficients of
$ (\bar \partial h)'(t) $ and $\bar\partial g^I(\bfz(t))$
 are still linearly dependent over $\mathbb R$.
\end{enumerate}
In the case of $(1)'$ and $(2)'$, we stop the operation. 
Otherwise we have $(3)'$ and 
we continue this operation till we get a modified gradient vector 
\[(\bar\partial h)^{(j)}(t)=\bar\partial h^I(\bfz(t))-\rho(t)\bar\partial g^I(\bfz(t)),\,\rho(t):=\sum_{i=1}^j b_i(t)
\]
for which either  its leading index is different from $\ga_g$ (case (1)) or the coefficients of the leading index are lenearly independent over $\mathbb R$ (case (2)).
Note that in this operation, the order of $(\bar\partial h)^{(j)}(t)$ is strictly increasing in $j$ while  
the matrix
$A(\bar\partial g,(\bar\partial h)^{(j)})$ is simply changed in  the second raw vector by
$\be(\bar\partial h^I)-\rho(0)\be(\bar\partial g^I)$. Therefore
$ (\bar\partial h)^{(j)}(t):=\bar\partial h^I(\bfz(t))-\rho(t) \bar\partial g^I(\bfz(t))$ satisfies 
\[
{\ord}\, (\bar\partial h)^{(j)}(t)\le d-p_1
\]
and therefore the operation should stop after finite steps, say $k$.
After the operation is finished, the normalized vector $((\bar\partial h)^{(k)}(t))_{norm}$ has linearly independent limit with that of $\bar\partial g^I(\bfz(t))$.
\newline\indent
Suppose further  $f^I_P(\bfa)=0$. This implies $g^I_P(\bfa)=h^I_P(\bfa)=0$.
 We will show now $\bfa_\infty=(a_1,\dots,a_k,0,\dots,0)$ is orthgonal to the limits of the normalized vectors
\[
v^g_\infty:=\lim_{t\to 0} (\bar\partial g^I(\bfz(t)))_{norm}\quad \text{and}\quad
v^{(\bar\partial h)^{(k)}}_\infty:=\lim_{t\to 0}((\bar\partial h)^{(k)}(t))_{norm}. 
\]
 First we consider  $v^g_\infty$.
 If $d_g<d-p_1$,  $j$-coefficient of $v^{g}_\infty$  is zero for $j\le k$ and 
$ \Re(v^g_\infty,\bfa_\infty)=0$ is obvious. 
If $d_g=d-p_1$,  we must have 
\[\frac{\partial g^I_P}{\partial \bar z_j}(\bfa)=0,\, k+1\le j\le m
\]
and the $i$-th coefficient of $v^g_\infty$ is
$\frac{\partial g^I_P}{\partial \bar z_i}(\bfa)$ for $1\le i\le k$ up to a scalor multiplication.
Thus the  assertion  follows from the Euler equality
\[
g^I_P(\bfa)=0=\sum_{i=1}^m p_i a_i \frac{\partial g^I_P}{\partial \bar z_i}(\bfa)= p_1\sum_{i=1}^k a_i \frac{\partial g^I_P}{\partial \bar z_i}(\bfa).
\] 
Now we consider $v^{(\bar\partial h)^{(k)}}_\infty$. 
We start from the equality
$h^I_P(\bfa)-\rho(0) g^I_P(\bfa)=0$. If $d_{ (\bar\partial h)^{(k)}(t)}<d-p_1$, $v^{(\bar\partial h)^{(k)} }_\infty$
and $\bfa_\infty$ are orthgonal by the same reason.
Suppose that  $d_{(\bar\partial h)^{(k)}(t)}=d-p_1$. Then we must have 
\[\frac{\partial h^I_P}{\partial \bar z_j}(\bfa)
-\rho(0)\frac{\partial g^I_P}{\partial \bar z_j}(\bfa) =0,\, k+1\le j\le m.
\]
Thus the assertion follows from the Euler equality of
the real valued radially weighted homogeneous polynomial
$h^I_P(\bfz,\bar\bfz)-\rho(0)g^I_P(\bfz,\bar\bfz)$.
The assertion (iv) can be shown in a similar way looking at the matrix $A$ before and after.
\end{proof}
\begin{Definition}Let $\bfz(t)$ be a analytic curve starting at the origin. Assume that 
$\bar\partial g(\bfz(t))\le \ord\,\bar \partial h(\bfz(t))$ and $\{\bar\partial g(\bfz(t)),\bar \partial h(\bfz(t))_{c(t)}\}$
(respectively  $\ord\,\bar\partial g(\bfz(t))>\bar \partial h(\bfz(t)))$ and 
$\{\bar\partial g(\bfz(t))_{c(t)},\bar \partial h(\bfz(t))\}$
)
is a good  modified gradient pair if 
they have linearly independent normalized limits over $\mathbb R$.

Take an arbitrary analytic curve $C(t): \bfz=\bfz(t)$ with $\bfz(0)=0$ and  a good modofied gradient pair, say $\{\bar\partial g(\bfz(t)),\bar \partial h(\bfz(t))_{c(t)}\}$
assuming $\bar\partial g(\bfz(t))\le \ord\,\bar \partial h(\bfz(t))$ for simplicity and suppose that  the following inequality is satisfied
for sufficiently small $t$, $0\le t\le \eps$, 
\begin{eqnarray}
{\ord}\,(\bar\partial g(\bfz(t),{\bar\bfz}(t)))/{\ord}\, \bfz(t)&\le &\theta\\
{\ord}\,(\bar\partial h(\bfz(t),{\bar\bfz(t)})_{c(t)}/{\ord}\, \bfz(t)&\le&\theta.
\end{eqnarray}
(If  $\bar\partial g(\bfz(t))> \ord\,\bar \partial h(\bfz(t))$, we exchange $g$ and $h$ in the above inequalities so that $\bar\partial g$ is to be modified.)
We define  the \L ojasiewicz exponent $ \ell_0(C(t))$ along an analytic curve $C(t)$ as 
the
infinimum of such $\theta$ satisfying the above inequality.
 The \L ojasiewicz exponent of $f$ is defined by the supremum of 
$\ell_0(C(t))$
 for all anaylytic curves $C(t)$.
\nl
These inequalities are equivalent to the inequality
\[
\|\bar\partial g(\bfz(t),{\bar\bfz}(t)))\|,
\|\bar\partial h(\bfz(t),{\bar\bfz}(t)))_{c(t)}\|\ge C\|\mathbf z(t)\|^\theta,\,\exists C>0
\]
for $0\le t\le \eps$.
\end{Definition}
Taking $c(t)\equiv 0$, such $\theta$ satisfies the usual \L ojasiewicz inequalities:
\begin{eqnarray}
\|\bar\partial g(\bfz,{\bar\bfz})\|&\ge &C \,\|\bfz\|^{\theta}\\
\|\bar\partial h(\bfz,{\bar\bfz})\|&\ge &C\,\|\bfz\|^\theta,\,\exists C>0.
\end{eqnarray}
in a sufficiently small neighbourhood of the origin.
% but the converse is not true.
%%%%%%

Now we are ready to generalize the results which are obtained  in previous sections for holomorphic functions.
%%%%%%%
\subsection{Convenient case}
We consider a strongly non-degenerate mixed function $f(\bfz,{\bar\bfz})$ with an isolated singularity at the origin.
A mixed function $f(\bfz,{\bar\bfz})$ is called {\em convenient} if for each $i=1,\dots, n$,
there is  a point  $B_i=(0,\dots,\overset i{\overset \smile b_i},\dots,0)$
on the Newton boundary $\Ga(f)$.   In the mixed function case, there might exist several corresponding  mixed monomial
$z_i^{\nu_i}{\bar z_i}^{\mu_i}$ with $\nu_i+\mu_i=b_i$ in the expansion of $f(\bfz,{\bar\bfz})$.
Such a monomial is called {\em an $i$-axis monomial}.
Let $B:=\max\{b_i\,|\, i=1,\dots, n\}$.
 An $i$ axis monomial
$z_i^{\nu_i}{\bar z_i}^{\mu_i}$
 is called {\em \L ojasiewicz monomial}  if $\nu_i+\mu_i=B$.
 A  \L ojasiewicz monomial $z_i^{\nu_i}{\bar z_i}^{\mu_i}$  is {\em  exceptional} if  there exists a monomial
$z_i^{\nu_i'}{\bar z_i}^{\mu_i'} w_j$ where $w_j=z_j$ or $\bar z_j$ in the expansion of $f(\bfz,{\bar \bfz})$
such that $\nu_i'+\mu_i' <B-1$.
\begin{Theorem}
 Let $f(\bfz,{\bar\bfz})$ % (respectively $f(\bfz,\bar\bfz)$)
 be a strongly nondegenerate convenient mixed function.
%(resp. mixed analytic function) and let $M$ be as above.
Then \L ojasiewicz exponent   $\ell_0(f)$ satisfies the inequality:
$\ell_0(f)\le B-1$.
Furthermore if $f$ has a \L ojasiewicz non-exceptional monomial, we have the equality $\ell_0(f)=B-1$.
\end{Theorem}
Using Lemma \ref{basic-lemma}, the proof is completely parallel to that of Theorem \ref{convenientLojasiewicz}.
%we list up the corresponding non-trivial parts. Basic tool is Lemma \ref{basic-lemma}.
%%%%%%%%%%%%
\subsection{Non-convenient mixed polynomials}
We consider the case of non-degenerate mixed polynomials with an isolated singularity at the origin. 
One point is how to define "{\em \L ojasiewicz non-degeneracy} " for mixed functions.

Let $\mathbb C^I$ be a vanishing coordinate subspace
and let $J$ be the complement of $I$.
%Let $\Xi\subset \Gamma_+(f)$ be an essentially non-compact face with $I(\Xi)=I$.
For each $i$, there must exist a mixed monomial $z_i^{n_{ij}}{\bar z_i}^{m_{ij}}w_j$ with $j\in J$ %or $z_i^{n_{ij}'}{\bar z_i}^{m_{ij}'}\bar z_j$
 as we have assumed that the origin is an isolated singularity. 
 Hereafter we use variable $w_j$ for either  $w_j=z_j$ or $w_j=\bar z_j$ for simplicity.
 We take $\ell_{ij}$ to be the minimum of 
 $\{n_{ij}+m_{ij}\}$ for fixed $i,j$. The point 
 $B_{ij}:=(0,\dots,\overset{\overset i\smile}\ell_{ij},\dots,\overset{\overset j\smile}1\dots,0)$
 is a point of $\Ga(f)$. We call a monomial  $z_i^{n_{ij}}{\bar z_i}^{m_{ij}}w_j$ 
  {\em an almost $i$-axis monomial} if $n_{ij}+m_{ij}=\ell_{ij}$.
 %or $n_{ij}'+m_{ij}'=\ell_{ij}$ respectively.
 Let $J_i$ be the set of $j$ for which such almost $i$-axis monomial exists and put $J(I)=\cup_{i\in I}J_i$ as in Part I.

Define $\ell(I):=\max\{\ell_{ij}\,|\, i\in I,\,j\in J_i\}$.
For our present purpose, we take the following definition of \L ojasiewicz non-degeneracy.
Consider a strictly positive weight vector  $Q\in N^{*I}$.
Consider $F(j):=z_j( \frac{\partial f}{\partial z_j})^I+\bar z_j( \frac{\partial f}{\partial\bar z_j})^I$.
Note that 
$\frac{\partial F(j)_{Q}}{\partial z_j}$ and
$\frac{\partial F(j)_{Q}}{\partial \bar z_j}$ 
are polynomials in $\mathbb C[z_1,\bar z_1,\dots, z_m,\bar z_m]$. 
Here $F(j)_{Q}:=(F(j))_Q$ with the weight of $z_j$ being $0$ by definition. 
We use the notations $f_j:= \frac{\partial f}{\partial z_j}$ and 
$f_{\bar j}:= \frac{\partial f}{\partial \bar z_j}$ for simplicity.
Note that 
\[\begin{split}
&\frac{\partial F(j)_{Q}}{\partial z_j}=\begin{cases}
(f_j)_Q,\,&\rdeg_Q f_j\le \rdeg_Q f_{\bar j}\\
0,\, &\rdeg_Q f_j > \rdeg_Q f_{\bar j}
\end{cases}\\
&\frac{\partial F(j)_{Q}}{\partial \bar z_j}=
\begin{cases}
(f_{ \bar j})_Q,\,&\rdeg_Q f_j\ge\rdeg_Q f_{\bar j}\\
0,\, &\rdeg_Q f_j<\rdeg_Q f_{\bar j}\end{cases}.
\end{split}
\]
Put $I':=\{i\,|\, q_i=m(Q)\}$ and $J(Q):=\cup_{i\in I'}J_i$.
$f(\bfz,\bar\bfz)$ is called {\em  \L ojasiewicz non-degenerate} if under any such situation and for any $ \bfa\in \mathbb C^{*I}$, 
there exists $i_0\in I', j_0\in J(Q)$ such that 
\begin{eqnarray}\label{Lojasiewicz-non-degeneracy2}
\left |\frac{\partial F(j_0)_{Q}}{\partial z_{j_0}}(\bfa)\right |\ne\left |\frac{\partial F(j_0)_{Q}}{\partial\bar z_{j_0}}(\bfa)\right |.
\end{eqnarray}
%For example, this is the case if one of the differential is vanishing.
Writing  $F(j):=g(j)+i h(j)$,
this is equivalent to  %the property:
\begin{Proposition}\label{mix-independence}
Under the above notations,
\begin{eqnarray}\label{good-modification}
\left \{\frac{\partial g(j_0)_{Q}}{\partial\bar z_{j_0}}(\bfa),\frac{\partial h(j_0)_{Q}}{\partial\bar z_{j_0}}(\bfa)
\right\}
\end{eqnarray}
are linearly independent over $\mathbb R$. %(Proposition 1 \cite{OkaPolar}).
\end{Proposition}
The assertion follows from Proposition 1 \cite{OkaPolar}.
Assume that such a $Q$ is associated with an analytic family
$\bfz(t)$ and $d=\rdeg_Q F(j)$. Then
$d\le\ell_{i_0,j_0}\, q_{i_0}$ and 
\[\begin{split}
&\frac{\partial g}{\partial \bar z_j}(\bfz(t))=\frac{\partial g(j)_Q}{\partial \bar z_j}(\bfa)t^d+\text{(higher terms)},\,
{\ord}\,\frac{\partial g}{\partial \bar z_j}(\bfz(t))=d\\
&\frac{\partial h}{\partial \bar z_j}(\bfz(t))=\frac{\partial h(j)_Q}{\partial \bar z_j}(\bfa)t^d+\text{(higher terms)},\,
{\ord}\,\frac{\partial h}{\partial \bar z_j}(\bfz(t))=d.
\end{split}
\]
\begin{Proposition}\label{vanishing}
 Using Proposition \ref{mix-independence}, there exists a good modified 
gradient pair $\{\bar\partial g(\bfz(t)),\bar\partial h(\bfz(t))_{c(t)}\}$
or $\{\bar\partial g(\bfz(t))_{c(t)},\bar\partial h(\bfz(t))\}$.
Their order in $t$ has a upper bound  $d=\rdeg_{Q} F(j_0)\le \ell_{i_0,j_0}q_{i_0}$. 
\end{Proposition}
 
 \subsection{Jacobian dual Newton diagram} We consider the derivatives $f_i(\bfz):=\frac{\partial f}{\partial z_i}(\mathbf z)$
 and $f_{\bar i}(\bfz):=\frac{\partial f}{\partial \bar z_i}(\mathbf z)$, $i=1,\dots, n$.
 Put $F_i(\bfz,\bar \bfz)=f_i(\bfz,\bar\bfz)f_{\bar i}(\bfz,\bar\bfz)$. If one of the derivatives vanishes identically, we consider only non-zero derivatives.
 For example, if $f_{\bar i}\equiv 0$,  we put $F_i=f_i$.
We consider their Newton boundary $\Ga(F_i),\,i=1,\dots, n$. 
%If $f_i\equiv 0$or $f_{\bar i}\equiv 0$, we omit this  derivative from the consideration.
 Two weight vectors $P,Q$ are {\em Jacobian equivalent} if
$\Delta(P,F_i)=\Delta(Q,F_i)$ 
for any $i=1,\dots,n$ and $\Delta(P,f)=\Delta(Q,f)$. 
We denote it by $P\underset J{\sim} Q$. This gives a polyhedral cone subdivision of $N_+$ and we denote this as 
$\Gamma_{J}^*(f)$ and we call it {\em the Jacobian dual Newton diagram of $f$}. $\Gamma_J^*(f)$ is a polyhedral cone subdivision of $N_+$
which is finer than $\Gamma^*(f)$.
 \newline\indent
Alternatively we can consider the function 
$F(\mathbf z)=f(\mathbf z)F_1(\mathbf z)\cdots F_n(\mathbf z)$. Then $\Gamma_J^*(f)$ is nothing but the dual Newton diagram $\Gamma^*(F)$
of $F$. For any weight vector $P$, we have $\Delta(P,F)=\Delta(P,f)+\Delta(P,F_1)+\dots+\Delta(P,F_n)$ where the sum is Minkowski sum.
%See \cite{Bo-Fe}.
For a weight vector $P$, the set of equivalent weight vectors in $\Gamma^*(f)$ and $\Gamma_J^*(f)$ is denoted as 
$[P]$ and $[P]_J$ respectively.
We consider the vertices of these dual Newton diagrams.
We denote the set of strictly positive vertices of $\Gamma^*(f)$ and $\Ga_J^*(f)$ by $\mathcal V^+,\,\mathcal V_J^+$ as before.
%Let $e_i=(0,\dots,\overset{\overset i\smile}1,\dots,0)$.
%%%%%%%%%
Now we can generalize Theorem 14.
Let $\mathcal V_J^{++}\subset \mathcal V_J^+$  be the set of the vertices of $\Gamma_J^*(f)$ which are in a vanishing boundary region   of $\Gamma^*(f)$
as in the holomorphic case.
%Thus if $P\in \mathcal V_J^+$, $d(P)>0$. % and $ \De(P)$ is a compact face of dimension $n-1$.
The numbers of $\mathcal V^+,\,\mathcal V_J^+,\,\mathcal V_J^{++}$ are finite. We define basic invariants, as before
\[\begin{split}
&\eta_{J,max}(f):=\max\{ \eta(P)\,|\, P\in \mathcal V^+\cup\mathcal V_J^{++}\},\\
&\eta_{max}:=\max\{\eta(P)\,|\, P\in \mathcal V^+\}\\
&\eta_{J,max}':=\max\{\eta_{k,i}'(R)\,|\, R\in  \mathcal V_J^{++},\,k,i=1,\dots,n\}\\
&\eta''_{J,max}(f):=\max\{\eta_{J,max}(f),\eta'_{J,max}(f)\},\,\text{where}\\
&\eta'_{k,i}(R):=\frac{\min\{d(R,f_i),d(R,f_{\bar i})\}}{m(R)}.
\end{split}\]

\begin{Theorem}\label{main-theorem2}
Let $f(\bfz)$ be a  non-degenerate, \L ojasiewicz non-degenerate mixed function  with an isolated singularity at the origin. 
Then \L ojasiewicz exponent $\ell_0(f)$ satisfies the estimation $\ell_0(f)\le \eta_{J,max}''(f)$.
\end{Theorem}
\begin{proof}
The proof is completely parallel to  that of Theorem \ref{main-theorem}.
We  consider an anlytic curve $C(t)$ parametrized as $\bfz(t)=(z_1(t),\dots, z_n(t))$ and 
put $I:=\{i\,|\, z_i(t)\not \equiv 0\}$. % and $$ be the complement of $J$.
Consider the Taylor expansion of $z_i(t)$ as before:
%%%%%%%%%%
\begin{eqnarray}
\begin{cases}
&\bfz(t)=(z_1(t),\dots, z_n(t)),\,\,\bfz(0)=0,\, \bfz(t)\in \mathbb C^{*I}\\
&z_i(t)=a_it^{p_i}+\text{(higher terms)},\quad i\in I\end{cases}
\end{eqnarray}
We put $P=(p_i)\in N_+^{*I}$ as before.
Assume first $I=\{1,\dots, n\}$. 
We divide the situation into  three cases as before.
\begin{enumerate}
\item[C-1] ${[P]}$ is an inner region. That is, $\overline{[P]}$  has only strictly positive weight vectors in the boundary.
\item[C-2] ${[P]}$ is  a regular boundary region.
% or a vanishing boundary region but ${[P]_J}$ is an inner region. % in $\mathcal V_J^+$.
\item[C-3] 
$[P]$ is a vanishing boundary region. In this case, we need to consider
the subdivision by $[P]_J$. There are three subcases.
\newline
C-3-1. ${[P]_J}$ is  an inner region.
%regular boundary region.
\newline C-3-2.
${[P]_J}$ is  a regular boundary region.
\newline
C-3-3. ${[P]_J}$ is also a vanishing boundary region.
\end{enumerate}
Then the proof goes exactly as that of Theorem \ref{main-theorem}, using  Lemma \ref{basic-lemma}.
For the cases C-1, C-3-1, C- 2,C-3-2,  we start from  a given good modified gradient pair 
and the estimation  of these gardient by  Lemma \ref{basic-lemma}.
Then  the argument is completely the same.
 We have the estimation $\ell_0(C(t))\le \eta_{J,max}''(f)$ in these cases.

For the case C-3-3,
consider the situation   that $P$  is not  a simplicially positive and 
 $R,Q$ as in Lemma \ref{nice-segment} so that $P$ is on the line segment $\overline{RQ}$,
 $R$ is simplicially positive and $Q$ is not strictly positive.
% If $Q$ is a non-vanishing weight, we proceed as the case C-3-2
%to get the estimation $\ell_0(C(t))\le \eta_{J,max}(f)$.
If $Q$ is non-vanishing, it reduced to Case 3-2.
Thus we assume that $Q$ is a vanishing weight vector and 
$d(Q,f)>0$. Assume that $Q=(q_1,\dots, q_n)$ and  $I=\{i\,|\, q_i=0\}$ and assume $I=\{1,\dots, m\}$ for simplicity.
Note that $\mathbb C^I$ is a vanishing coordinate subspace.
For each $i\in I$, there exists some $j\not\in I$ and a monomial $z_i^{n_{i,j}}z_j$ with a non-zero coefficient, as $f$ has an isolated singularity at the origin.
Put $J_i$ be the set of such $j$ for a fixed $i\in I$ and put $J(I)=\cup_{i\in I}J_i$.
Here $n_{i,j}$ is assumed to be the smallest when $j$ is fixed.
Put $\xi_{I}:=\max\{n_{i,j}\,|\, i\in I,j\in J_i\}$ and $\xi(f)$ be the maximum of $\xi_{I}$ where $I$ moves in the coordinate subspaces corresponding to  vanishing coordinate subspaces.
Put $\eta'_{J,max}(f):=\max\{\eta'_{k,i}(R)\,|\,R\in  \mathcal V^{++}\}$
where $\eta_{j,i}'(R)={d(R,f_i)}/{r_j}$. Under the above situation, 
we will prove, as in the holomorphic case,  that  
\begin{eqnarray*}
(\star)\quad \ell_0(C(t))\le\max\{\xi(f),\eta_{J,max}(f),\eta'_{J,max}\}.\
\end{eqnarray*}
%by  the induction of $\dim\,[P]$.
%Put$\tilde J$ be the union of $J$ and such $j$.
Consider the normalized weight vector
$\hat R_s:=(1-s)\hat R+s\hat Q,\,0\le s\le 1$. 
Note that $\hat R_0=\hat R,\,\hat R_1=\hat Q$ and putting 
$\hat R_s=(\hat r_{s,1},\dots,\hat r_{s,n})$,
\[\begin{split}
%&\hat R_s=(\hat r_{s,1},\dots,\hat r_{s,n}),\, \hat R_0=\hat R,\,\hat R_1=\hat Q\\
&\hat r_{s,i}=
\begin{cases} &(1-s)\hat r_i,\quad 1\le i\le m\\
&(1-s)\hat r_{i}+s\hat q_{i},\quad m<i\le n.
\end{cases}\\
\end{split}
\]
 Put  $I'=\{i\in I\,|\, \hat r_i=m(\hat R_I)\}$
 and $J'=\cup_{i\in I'}J_i$.  
 Thus for $i\le m$, the normalized weight $\hat r_{sj}$ goes to $0$, when  $s$  approaches to $ 1$. On the other hand,  for  $j>m$,  $\hat r_{sj}\ge \delta, 0\le \forall s\le 1$ for some $\delta>0$.
 Thus there exists an $\eps,\,1>\eps>0$ so that for
 $1-\eps\le s\le 1$, $m(\hat R_s)$ is taken by $i\in I'$.
 Note that for $j\in  J'$,  there exists a  small enough $\eps_2,\,\eps_2\le \eps_1$ so that  $(f_j)_{\hat R_s}=((f_j)^I)_{\hat R_{s}}$ as 
 $\hat r_{sj}\ge \de $ for $j>m$ for 
  $1-\eps_2\le s\le 1$.  Here $(f_j)^I$ is the restriction of $f_j$ to $\mathbb C^I$ and 
 $(\hat R_{s})_I$ is the $I$ projection of $\hat R_{s}$ to $N_+^I$.
 %$(\hat R_{s})_I\in {N_+^I}$.
  That is, $(f_j)_{\hat R_s}$  contains only variable $z_1,\dots, z_m$.
By the \L ojasiewicz non-degeneracy, there exists $i_0\in I'$ and $j_0\in J_{i_0}$
 such that 
 %the property:
\begin{eqnarray}\label{good-modification}
\left \{\frac{\partial g(j_0)_{Q}}{\partial\bar z_{j_0}}(\bfa),\frac{\partial h(j_0)_{Q}}{\partial\bar z_{j_0}}(\bfa)
\right\}
\end{eqnarray}
are linearly independent over $\mathbb R$.
 Here we use  the same notation as in (\ref{good-modification}).
 %  by Lemma\,\ref{Lojasiewicz-non-degeneracy}.
 By the definition of Jacobian dual Newton diagram and Poroposition \ref{good-modification},  there is a good modified gradient pair, say 
 $\bar\partial g(\bfz(t)),\bar \partial h (\bfz(t))_{c(t)}$ (we assume
 $\ord\,\bar\partial g(\bfz(t))\le\ord\,\bar\partial h(\bfz(t))$ for simplicity)
 so that {\em  their orders are estimated from above by
 $\ell_{i_0,j_0}q_{i_0}$. }
 
 The rest of the argument is simply the evaluation of the number
 $\ell_{i_0,j_0}q_{i_0}/m(\hat R_{s_0})$ and 
the proof is completed  by the exact same argument  as in the proof of Theorem \ref{main-theorem}, Case 3-3-3  using the following.
The case $C(t)\in \mathbb C^{*I}$ with $I^c\ne \emptyset$ is treated also by the exactly same argument. \end{proof}
%Lemma \ref{estimation-vanishing} as
\begin{Lemma} (Restatement of Lemma \ref{estimation-vanishing} )
We have the inequality:
$\ell(I)\le \eta_{max}(f)$.
\end{Lemma}

 Theorem  \ref{estimation-weighted} for non-degenerate,  \L ojasiewicz non-degenerate  (radially)  weighted homogeneous polynomial and
\L ojasiewicz Join Theorem \ref{Lojasiewicz Join Theorem} also hold in the exactly same way for mixed functions. For example,  we can state
\begin{Theorem}\label{mixed weighted}
Let $f(\bfz)$ be a strongly non-degenerate, \L ojasiewicz non-degenerate mixed weighted homogeneous polynomial with isolated singularity at the origin
and $\dim\,\Ga(f)=n-1$. Let $R$ be the weight vector of $f$. Then we have the estimation
$\ell_0(f)\le \eta(R)$.
\end{Theorem}
%%%%%%%%%%%
\subsection{Making $f$ convenient}
We also generalize Theorem 21. Take an integer $N>\eta''_{J,max}(f)+1$.
Consider mixed polynomial $R(\bfz,\bar \bfz):=\sum_{i=1}^n c_i z_i^{m_i}{\bar z_i}^{n_i}$ where $n_i,m_i$ are any fixed non-negative integers with 
$m_i+n_i=N_i\ge N$ and $n_i\ne m_i$.
We choose such $\{(m_i,n_i)\,|\, i=1,\dots,n\}$ and fix them. The coefficients
$c_1,\dots, c_n$ are generic so that 
$f_1:=f(\bfz,\bar\bfz)+R(\bfz,\bar \bfz)$ is strongly non-degenerate.
Consider the family $f_s(\bfz,\bar\bfz)=f(\bfz,\bar\bfz)+sR(\bfz,\bar \bfz)$.
Then we have the following.
\begin{Theorem}
There exists a $r_0>0$ such that for any $r\le r_0$, the sphere $S_r$ and the family of  hypersurface $V_s:=f_s\inv(0)$ intersect transversely for any 
$0\le s\le 1$. In particular, the links of $f$ and $f_1$ are isotopic and their Milnor fibrations are isomorphic.
\end{Theorem}
The proof is similar and we leave it to the reader.

%There is a similar result in Fukui \cite{Fukui}.
%%%%%%%%%%%%%%%%%%%%%%%%%%%%%%%%%%%%%%%%%%%%%
%\newpage
%\bibliographystyle{abbrv}
%\bibliography{./okabib}
%\begin{comment}
\def\cprime{$'$} \def\cprime{$'$} \def\cprime{$'$} \def\cprime{$'$}
  \def\cprime{$'$} \def\cprime{$'$} \def\cprime{$'$} \def\cpri{$'$}

%\end{comment}
\end{document}